\documentclass[10pt]{article}

\makeatletter
\DeclareRobustCommand{\qed}{%
  \ifmmode 
  \else \leavevmode\unskip\penalty9999 \hbox{}\nobreak\hfill
  \fi
  \quad\hbox{\qedsymbol}}
\newcommand{\openbox}{\leavevmode
  \hbox to.77778em{%
  \hfil\vrule
  \vbox to.675em{\hrule width.6em\vfil\hrule}%
  \vrule\hfil}}
\newcommand{\qedsymbol}{\openbox}
\newenvironment{proof}[1][\proofname]{\par
  \normalfont
  \topsep6\p@\@plus6\p@ \trivlist
  \item[\hskip\labelsep\itshape
    #1.]\ignorespaces
}{%
  \qed\endtrivlist
}
\newcommand{\proofname}{Proof}
\makeatother

\usepackage[utf8]{inputenc}

\usepackage{bm}

\usepackage{color}
\usepackage{latexsym}
\usepackage{dsfont}
\usepackage{amssymb}
\usepackage{comment}
\usepackage{graphicx}
\usepackage{amsmath,amsfonts,amssymb,theorem,euscript,array,enumerate,amsfonts,mathrsfs}
\usepackage{hyperref}
\usepackage{appendix}
\usepackage[T1]{fontenc}
\usepackage{babel}
\numberwithin{equation}{section}
\usepackage{bbm}
\usepackage{subfigure}
\usepackage{color}
\usepackage{stmaryrd}

\usepackage[algo2e,ruled,vlined]{algorithm2e} 
 \usepackage{algorithm}
 \usepackage{algorithmicx}
\usepackage{algpseudocode}
\usepackage{ marvosym }
\usepackage{hyperref}
\usepackage{bm}
\usepackage{xcolor, soul}

\usepackage{mathtools}
\mathtoolsset{showonlyrefs}

\usepackage{comment}

\usepackage{tikz}
\usetikzlibrary{fit,matrix,chains,positioning,decorations.pathreplacing,arrows}

\usepackage{mathtools}

\usepackage[a4paper,margin=1in,heightrounded]{geometry}

\usepackage{algpseudocode}
\usepackage{algorithm}

\usepackage[normalem]{ulem}

\def \b1{\bf{1}}

\def \I{\mathbb{I}}
\def \N{\mathbb{N}}
\def \R{\mathbb{R}}

\def \E{\mathbb{E}}
\def \F{\mathbb{F}}

\def \P{\mathbb{P}}

\def \S{\mathbb{S}}
\def \W{\mathbb{W}}

\def \bd{\mathbb{d}}

\def \bd{\boldsymbol{d}} 
 
\def \bV{\boldsymbol{V}}

\def \balpha{\boldsymbol{\alpha}} 
\def \bbeta{\boldsymbol{\beta}}

\def \mra{\mathrm{a}} 
 
\def \mrx{\mathrm{x}}
\def \mry{\mathrm{y}}
\def \mrz{\mathrm{z}}

\def \d{\mathrm{d}}

\def\esssup_#1{\underset{#1}{\mathrm{ess\,sup\, }}}

\def\argmin_#1{\underset{#1}{\mathrm{argmin\, }}}
\def\argmax_#1{\underset{#1}{\mathrm{argmax\, }}}

\def\dmu{\frac{\delta}{\delta m}}
\def\dm#1{\frac{\delta}{\delta m}}

\def \Ac{{\cal A}}
\def \Bc{{\cal B}}
\def \Cc{{\cal C}}

\def \Ec{{\cal E}}
\def \Fc{{\cal F}}

\def \Hc{{\cal H}}
\def \Ic{{\cal I}}

\def \Lc{{\cal L}}
\def \Pc{{\cal P}}

\def \Nc{{\cal N}}

\def \Sc{{\cal S}}

\def \Uc{{\cal U}}

\def \Wc{{\cal W}}

\def \Ik{\mathfrak{I}}

\def \eps{\varepsilon}

\def\reff#1{{\rm(\ref{#1})}}

\def\bX{{\bf X}}
\def\bY{{\bf Y}}
\def\bZ{{\bf Z}}

\def\bd{{\bf d}}

\def \mra{\mathrm{a}}

\def \d{\mathrm{d}}

\def \mry{\mathrm{y}} 
 
\def \mrz{\mathrm{z}} 
\def \mrv{\mathrm{v}}

\def\beqs{\begin{eqnarray*}}
\def\enqs{\end{eqnarray*}}
\def\beq{\begin{eqnarray}}
\def\enq{\end{eqnarray}}

\newcommand{\red}[1]{\textcolor{red}{#1}}

\def\red#1{{\color{red}#1}}

\newtheorem{Theorem}{Theorem}[section]
\newtheorem{Definition}{Definition}[section]
\newtheorem{Proposition}{Proposition}[section]
\newtheorem{Assumption}{Assumption}[section]
\newtheorem{Lemma}{Lemma}[section]
\newtheorem{Corollary}{Corollary}[section]
\newtheorem{Remark}{Remark}[section]
\newtheorem{Example}{Example}[section]
\numberwithin{equation}{section}

\title{Stochastic maximum principle for optimal control problem of non exchangeable mean field systems}

\author{Idris Kharroubi \footnote{LPSM, Sorbonne University and  Universit\'e Paris Cit\'e, \sf kharroubii at lpsm.paris. 
The work of this author is  partially supported by Agence Nationale de la Recherche (ReLISCoP grant ANR-21-CE40-0001). } 
\and 
Samy Mekkaoui\footnote{Ecole Polytechnique, CMAP, \sf samy.mekkaoui at polytechnique.edu; This author is supported by the S-G Chair "Risques Financiers", and the "Deep Finance and Statistics" Qube-RT Chair.}
\and 
Huy\^en Pham\footnote{Ecole Polytechnique, CMAP, \sf huyen.pham at polytechnique.edu; This author is supported by
the BNP-PAR Chair “Futures of Quantitative Finance", the Chair “Risques Financiers", and by FiME,
Laboratoire de Finance des Marchés de l’Energie, and the “Finance and Sustainable Development” EDF -
CACIB Chair.}}

\date{}

\begin{document}

\maketitle

\begin{abstract}
We study the Pontryagin maximum principle by deriving necessary and sufficient conditions for a class of optimal control problems arising in non exchangeable mean field systems, where agents interact through heterogeneous and asymmetric couplings. Our analysis leads to a collection of 
forward-backward stochastic differential equations (FBSDE)  of non exchangeable mean field type. Under suitable assumptions, we establish the solvability of this system. As an illustration, we consider the linear-quadratic case, where the optimal control is characterized by an infinite dimensional system of Riccati equations.  
\end{abstract}

\vspace{5mm}

\noindent {\bf MSC Classification}: 49N80; 60K35; 93E20.   

\vspace{5mm}

\noindent {\bf Key words}: Pontryagin maximum principle;  mean-field FBSDE; heterogeneous interactions; non-exchangeable systems; linear-quadratic graphon optimal control.
\newpage
\tableofcontents
\section{Introduction}

The study of large population and complex systems is a central question in mathematical modeling with wide-ranging  applications in modern society, including social networks, power grid infrastructures, financial markets, and lightning networks. The analysis of such systems has given rise to models in which the representative state of an agent depends on its own distribution, leading to the class of McKean-Vlasov (MKV) equations. These equations have  been extensively studied  in the context of mean-field game (MFG) or mean-field control (MFC). We refer the reader to the seminal lectures of P.L. Lions \cite{lasry2007mean} and the comprehensive monographs by Carmona and Delarue \cite{carmona_probabilistic_2018,carmona_probabilistic_2018-1}, which introduce the foundational mathematical tools for these problems, such as  It\^o's calculus along flow of probability measures, stochastic maximum principle, MKV-type FBSDEs, and the Master Bellman equation in the Wasserstein space.

In this paper, motivated by the modeling of more complex networks, we focus on large-scale systems of agents with  asymmetric and heterogeneous interactions, leading to non exchangeable systems. Recent years have witnessed a growing interest in such models, notably via the graphon framework introduced by  Lovász in \cite{lovasz_large_2010}, and later explored in the context of heterogeneous mean-field systems in \cite{coppini2024nonlinear,bayraktar2023graphon,jabin2025mean,lacker2023label}. 
More recently, in the setting of optimal control, new general frameworks have emerged where the interactions are modelled via the full collection of marginal laws of all agents without necessarily assuming a specific graphon-type structure. 
This generalization leads to the study of a new class coupled systems with non-exchangeable interactions 
\begin{align}\label{eq : dynamics X}
\d X_t^u &= \;  b(u,X_t^u,(\P_{X_t^v})_{v \in I} ,\alpha_t^u) \d t  + \sigma(u,X_t^u,(\P_{X_t^v})_{v \in I}, \alpha_t^u) \d W_t^u, \quad 
u\in I = [0,1], \; t \in [0,T].  
\end{align}
In the context of MFC, we are then looking for a collection of control processes $\alpha = (\alpha^u)_{u \in I}$ to minimize the following cost functional over a finite horizon $T$ 
\begin{equation}\label{eq : functional J}
J(\alpha) = \int_{I} \E \Big[\int_{0}^{T} f(u,X_t^u,(\P_{X_t^v})_{v},\alpha_t^u) \d t + g \big(u,X_T^u, (\P_{X_T^v})_{v} \big) \Big] \d u.
\end{equation}
Such control problem has recently been studied by a dynamic programming approach in  \cite{de2024mean}. 
This leads to  a master equation on the space $L^2(\Pc_2(\R^d))$, consisting of collections $(\mu^u)_{u \in I}$ of square integrable probability measures on $\R^d$ s.t $\int_{I}\int_{\R^d} |x|^2 \mu^u(dx) du < + \infty$. 
Since the Pontryagin maximum principle provides a natural and complementary method for analyzing stochastic control problems \cite{peng_general_1990,yong_stochastic_1999}, this motivated our investigation of its formulation in the 
context of non-exchangeable mean field systems.  Due to the heterogeneity among agents, encoded in the collection 
of SDEs \eqref{eq : dynamics X}, we are naturally led, to study a family of FBSDEs indexed by $u$ $\in$ $I$, in contrast to classical mean-field models that can be summed up into a single FBSDE (see \cite{carmona_forwardbackward_2015}). 
Such a family of FBSDEs has been considered by \cite{bayraktar2023propagation} in the context of linear graphon-based interactions. 
Very recently, and independently of our current work, the paper \cite{cao_probabilistic_2025} has extended 
this study from a control perspective by Pontryagin maximum principle in the case of nonlinear graphon-based interactions, i.e., the coefficients $b,\sigma$ depend on the collection of laws only through a graphon-weighted sum.  In this work, we address the general case of non exchangeable systems as in 
\eqref{eq : dynamics X}.

A particularly notable subclass of optimal control problems within the Pontryagin maximum principle framework arises when the controlled dynamics is linear and the cost functional is quadratic. In the classical mean-field setting, existence and uniqueness results for the corresponding FBSDE systems are well established (see \cite{carmona_forwardbackward_2015}). This leads to the so-called  linear-quadratic mean-field (LQ-MF) control framework, which has been widely studied due to its analytical tractability and broad applicability  \cite{basei2019weak,Yong2013,carmona_mean_2015}. More recently, the work \cite{de2025linear} extended this theory to the linear-quadratic non-exchangeable mean-field (LQ-NEMF) setting, where heterogeneity in interactions is explicitly modeled. This generalization introduces a new class of infinite-dimensional Riccati equations that capture the structure of non-exchangeable interactions and opens the door to novel applications.


\paragraph{Our work and contributions.}
In the spirit of classical mean-field control, we formulate the Pontryagin maximum principle for non-exchangeable systems of interacting agents, where the state process $X=(X^u)_{u \in I}$ evolves according to the dynamics 
\eqref{eq : dynamics X}, and the optimality criterion is given by the cost functional \eqref{eq : functional J}.   
This leads us to introduce the associated Hamiltonian function
\begin{align}\label{eq : Hamiltonian def}
H(u,x,(\mu^v)_v,y,z,a) &:= \;  b(u,x,(\mu^v)_v,a) \cdot y + \sigma(u,x,(\mu^v)_v , a):z + f(u,x,(\mu^v)_v,a),
\end{align}
where $(Y,Z) = (Y^u,Z^u)_{u \in I}$ are solution to the following collection of  adjoint equations

\begin{equation}\label{eq : adjoint equations Intro}
        \left\{\begin{aligned} 
        \d Y_t^u &= \; - \partial_{x} H(u,X_t^u,(\P_{X_t^v})_{v},Y_t^u,Z_t^u,\alpha_t^u) dt  + Z_t^u 
        \d W_t^u \\
        & \qquad - \;  \int_{I} \mathbb{\tilde{E}}\Big[ \partial \frac{\delta }{\delta m}H(\tilde{u},\tilde{X}_t^{\tilde{u}},(\P_{X_t^v}),\tilde{Y}_t^{\tilde{u}},\tilde{Z}_t^{\tilde{u}},\tilde{\alpha}_t^{\tilde{u}})(u,X_t^u)\Big] \d\tilde{u} \d t , \\
        Y_T^u &= \; \partial_{x} g \big(u,X_T^u, (\P_{X_T^v})_{v} \big) + \int_{I}  \mathbb{\tilde{E}} \Big[\partial \frac{\delta }{\delta m}g \big(\tilde{u},\tilde{X}_T^{\tilde{u}},(\mathbb{P}_{X_T^v})_{v}\big)(u,X_T^{u}) \Big] \d\tilde{u}, \\
        \end{aligned}
        \right.
\end{equation}
where $(\tilde{X},\tilde{Y},\tilde{Z},\tilde{\alpha})$ is an independent copy of $(X,Y,Z,\alpha)$ defined on another probability space $(\tilde{\Omega},\tilde{\Fc},\tilde{\P})$.
Here $\frac{\delta}{\delta m}$ and $\partial \frac{\delta}{\delta m}$ stand for the linear functional derivative and his gradient and will be precised in the sequel.

Our objective is to characterize the optimal control via a pointwise minimization of the Hamiltonian $H$. To this end, we introduce a suitable notion of convexity for functions defined on  the space $L^2 (\Pc_2(\R^d) )$ and assume this property for both $H$  and the terminal cost $g$. Under these assumptions, we establish a necessary and sufficient condition for a control $\alpha=(\alpha^u)_{u \in I}$ to be optimal. The characterization, in turn, leads to the study of a collection of FBSDEs of  the following form 
\begin{equation}\label{eq : FBSDE equations}
        \left\{\begin{aligned} 
        \d X_t^u &= \;   b(u,X_t^u,(\P_{X_t^v})_{v}, \hat{\alpha}_t^u ) \d t 
        + \sigma(u,X_t^u,(\P_{X_t^v})_{v}, \hat{\alpha}_t^u)  \d W_t^u, \\
        dY_t^u &= \;  -\bigg[ \partial_{x} H(u,X_t^u,(\mathbb{P}_{X_t^v})_{v},Y_t^u,Z_t^u,\hat{\alpha}_t^u)   + \int_{I} \mathbb{\tilde{E}}\Big[ \partial \frac{\delta }{\delta m}H(\tilde{u},\tilde{X}_t^{\tilde{u}},(\mathbb{P}_{X_t^v})_{v},\tilde{Y}_t^{\tilde{u}},\tilde{Z}_t^{\tilde{u}},\tilde{\hat{\alpha}}_t^{\tilde{u}})(u,X_t^u)\Big] \d \tilde{u} \bigg] \d t  \\
        &\qquad + \;  Z_t^u \d W_t^u, \\
        X_0^u &= \; \xi^u , \\
        Y_T^u &= \;  \partial_x g(u,X_T^u, (\mathbb{P}_{X_T^v})_{v}) + \int_{I}\mathbb{\tilde{E}} \Big[\partial \frac{\delta }{\delta m}g (\tilde{u}, \tilde{X}_T^{\tilde{u}},(\mathbb{P}_{X_T^v})_{v})(u,X_T^u) \Big] \d\tilde{u},
        \end{aligned}
        \right.
\end{equation}
where $\hat{\alpha}_t^u = \hat{\mra}(u,X_t^u,(\P_{X_t^v})_{v}, Y_t^u,Z_t^u)$ for a Borel measurable function 
$\hat{\mra}$ defined on $I \times  \R^d \times L^2 (\Pc_2(\R^d)) \times \R^d \times \R^{d \times n}$. 
Under standard assumptions on the model coefficients, and by introducing a suitable  solution space for the collection of FBSDE \eqref{eq : FBSDE equations}, we establish the existence and uniqueness of solutions. We would like to emphasize that the  well posedness of the control problem and the related family of FBSDEs require a suitable definition of admissible controls, specifically in terms of measurability with respect to the label. 
In particular, we pay a careful attention to check that the derived optimal controls satisfy the admissibility condition. Finally, we illustrate our methodology in the linear-quadratic setting and prove the consistency of our stochastic maximum principle  with the dynamic programming approach (see \cite{de2025linear}).


\paragraph{Outline of the paper.}
The structure of the paper is as follows. In Section 2, we formulate the mean-field control problem for non-exchangeable systems. Section 3 introduces key preliminaries on functions defined over the space $L^2 (\Pc_2(\R^d))$, which are essential for the subsequent analysis. In Section 4, we derive necessary and sufficient optimality conditions using the stochastic maximum principle, under suitable regularity and convexity assumptions. This leads to the study of a collection of FBSDE of non-exchangeable mean-field type. Section 5 is devoted to establishing existence and uniqueness results for this system of FBSDE under appropriate technical conditions. Finally, in Section 6, we illustrate our approach in the linear-quadratic setting.

\paragraph{Notations.}
\begin{itemize}
\item We denote by $x\cdot y$ the scalar product between vectors $x,y$, and by $A:B$ $=$ ${\rm tr}(A B^{\top})$ the inner product of two matrices $A,B$ with compatible dimensions, where $B^{\top}$ is the transpose matrix of $B$. The Euclidean norm of a vector will be denoted by $ |\cdot |$.
\item  Throughout the paper, $T>0$ denotes a fixed time horizon. For any integer $\ell\geq1$ and $[a,b]\subset [0,T]$ the space $C([a,b];\R^\ell)$ of continuous functions $[a,b]\to \R^\ell$ will be denoted simply $C_{[a,b]}^\ell$. It is endowed with the supremum norm and the corresponding distance and Borel $\sigma$-algebra. For $w\in C^\ell_{[a,b]}$ and $[a',b']\subset [a,b]$,  $w_{[a',b']}\in C^\ell_{[a',b']}$ stands for the restriction of $w$ to $[a',b']$.    
We also denote by
$\mathbb{W}_T$ the Wiener measure on $C^\ell_{[0,T]}$.
\item  For $a,b,c\in[0,T]$ such that $a\leq b\leq c$, $\hat{w}\in C^d_{[a,b]}$ and $\check{w}\in C^d_{[b,c]}$  we define the concatenation $\hat{w}\oplus \check{w} \in C^d_{[a,c]}$ by the formula
\begin{equation}\label{DefConc}
(\hat{w}\oplus \check{w})(s)=\left\{ \begin{array}{ll}
\hat{w}(s)& \text{ if } s\in [a,b],
\\
\hat{w}(b)+\check{w}(s)-\check{w}(b)& \text{ if } s\in [b,c].
\end{array}\right.
\end{equation}
\item  For any generic Polish space $E$ with a complete metric $d$,
we denote by $\Pc_2(E)$ the Wasserstein space of Borel probability measures $\rho$ on $E$ satisfying 
$\int_E d(x,0)^2\,\rho(\d x)<\infty$, 
where $0$ denotes an arbitrary fixed element in $E$ (the origin when $E$ is a vector space).  $\Pc_2(E)$ is
endowed with the $2$-Wasserstein distance $\Wc_2$ corresponding to the quadratic transport cost $(x,y)\mapsto d(x,y)^2$. As a measurable space, $\Pc_2(E)$ is 
endowed with the corresponding Borel $\sigma$-algebra.  
\item Given a random variable $Y$ on a probability space $(\Omega,\Fc,\P)$, we denote by $\P_{Y}$ the law of $Y$ under $\P$. We shall also use the notation  $\Lc(Y)$ for the law of $Y$ (under $\P$) when there is no ambiguity.
\item [$\bullet$] We denote by $I$ the interval $(0,1)$ and by $L^2(I; \Sc^{d})$ the set of collection of processes $Y=(Y^u)_{u \in I}$ such that $Y^u$ is $\F^u$-adapted valued in $\R^d$, continuous and 
\begin{align}
    \lVert Y \rVert_{L^2(I;\Sc^d)}:= \Big(\int_{I} \E \big[\underset{0 \leq t \leq T}{\text{ sup }} |Y_t^u|^2 \big] \d u\Big)^{\frac{1}{2}} < + \infty .
\end{align}
Similarly, we denote by $L^2(I; \Hc^{2,d \times n})$ the set of collection of processes $Z=(Z^u)_{u \in I}$ such that $Z^u$ is $\F^u$-adapted valued in $\R^{d \times n}$ and 
\begin{align}
    \lVert Z \rVert_{L^2(I;\Hc^{2, d \times n})} := \Big(\int_{I} \E \big[\int_{0}^{T} | Z_t^u|^2 \d t \big] \d u \Big)^{\frac{1}{2}} < + \infty .
\end{align}
\item [$\bullet$] For a Banach space $(E, |.|_{E})$ endowed with $\Bc(E)$, we denote by $L^2(I;E)$ and  $L^{\infty}(I;E)$ the set of elements $\phi = (\phi^u)_u$ s.t $u \mapsto \phi^u \in E$ is measurable and $\int_{I} |\phi^u|^2_{E} du <  + \infty$, resp $\underset{u \in I}{\text{ sup }}  |\phi^u|_{E} < + \infty$. We also denote by $L^2(I \times I;E)$, resp $L^{\infty}(I \times I; E)$ the set of measurable functions $G : I \times I \to E $, s.t. $\int_{I} \int_{I} |G(u,v)|^2_{E} du dv < + \infty$, resp $\underset{u,v \in I \times I}{\text{ sup }} |G(u,v)|_{E} < + \infty$. 

\item [$\bullet$] For any integer $q \in \N^*$, we denote by  $\S^q$ (resp $\S^q_{+}$, $\S^q_{> +}$) the set of $q$-dimensionnal symmetric  (resp positive semi-definite, positive definite) matrices.

\end{itemize}


\section{Control of mean field non exchangeable SDE
} \label{secMFC}

\subsection{Heterogeneous population}

We recall that $I$ is the interval $(0,1)$. 
The interval $I$ represents the continuum of heterogeneous agents and is endowed with its Borel $\sigma$ algebra together with Lebesgue measure. To take into account  the heterogenity of a population we shall  consider mappings $u \mapsto \mu^u$ from $I$ to $\Pc_2(E)$. 

To check measurability of such maps, we will use the fact that the Borel $\sigma$-algebra in $\Pc_2(E)$ coïncides with the trace of the Borel $\sigma$-algebra corresponding to the weak topology, and that measurability holds iif the maps of the form 

\begin{equation}\label{eq : Measurability condition}
    u \in I \mapsto \int_{E} \Phi(x) \mu^u(\d x) \in \R, 
\end{equation}
are measurable for every choice of bounded continuous functions $\Phi : E \to \R$.

The space $L^2 \big(\Pc_2(E) \big)$ consists of elements $\mu=(\mu^u)_{u \in I}$ that are measurable functions $I \to \Pc_2(E)$ satisfying 
\begin{equation}
    \int_{I} \int_{\R^d} d(x,0)^2 \mu^u(\d x) \d u  = \int_{I} W_2(\mu^u,\delta_{0})^2 \d u  < \infty,
\end{equation}
where $\delta_0$ denotes the Dirac mass at the fixed element $0$. Morever, the space $L^2 (\Pc_2(E))$ is endowed with the complete metric 
{\begin{equation}\label{eq : distance collection mesures}
    \bd(\mu,\nu) := \left(\int_{I} W_2(\mu^u,\nu^u)^2 \d u \right)^{1/2}, \quad \mu,\nu \in L^2\big(\Pc_2(E)\big).
\end{equation}
}

For a collection $\xi = (\xi^u)_{u \in I}$ of random variables defined on a probability space $(\Omega,\Fc,\P)$, we denote by $\P_{\xi^.}$ the collection of probability laws $(\P_{\xi^u})_{u \in I}$.

We denote by $A$ the set of control actions and we assume that it is a convex subset of $\R^m$ with $m \in \N^*$.


\subsection{Controlled mean field non exchangeable SDE}

Let $(\Omega,\Fc,\P)$ be a complete probability space. For every $u \in I$, we are given an $\R^n$-valued standard Brownian motion $W^u=(W^u_t)_{0 \leq t \leq T}$ and an independent random variable $U^u$ having uniform distribution in $(0,1)$. We assume that 
{$(W^u,U^u)$ and $(W^v,U^v)$ are independent for $u,v\in I$ with $u\neq v$}.
For every $u \in I$, we denote by $(\Fc_t^{W^u})_{t \in [0,T]}$  the natural filtration generated by $W^u$
 and by $\F^u =(\Fc_t^u)_{t \in [0,T]}$ the filtration given by  

\begin{align}
    \Fc_t^u := \Fc_t^{W_t^u} \vee \sigma(U^u) \vee \Nc,\quad t \in [0,T],
\end{align}
where $\Nc$ is the family of $\P$-null sets.

\noindent We fix drift and diffusion functions $b,\sigma:~I\times \R^d\times L^2(\Pc_2(\R^d) )\times A\rightarrow \R^d,\R^{d\times n}$, on which we make the following assumption.
\begin{Assumption}\label{assumptionzero}
The functions $b$ and $\sigma$ are Borel measurable. There exist constants $L\ge0$, $M\ge 0$ such that
\begin{itemize} 
\item[(i)] $$ |b(u,x,\mu,a)-b(u,x',\mu',a)| \le L\big(|x-x'|+ \bd(\mu,\mu') \big),$$
 \item[(ii)] 
 $$
|\sigma(u,x,\mu,a)-\sigma(u,x',\mu',a)|
 \le L\big(|x-x'|+ \bd(\mu,\mu') \big),
 $$
\item[(iii)] 
$$
|b(u,0,\delta_0,a)|+
|\sigma(u,0,\delta_0,a)|
 \le M\big( 1+|a|\big),$$
\end{itemize} 
 for every  $u\in I$, $x,x'\in\R^d$, $\mu,\mu'\in L^2 (\Pc_2(\R^d))$, $a\in A$.
\end{Assumption}


We consider  a class of mean-field non exchangeable controlled systems described by a collection of controlled state process $X = (X^u)_{u \in I}$ satisfying the following (coupled) SDEs
\begin{align}\label{eq : control process X}
    \left\{\begin{aligned}
        \d X_t^u &= b(u,X_t^u,\P_{X_t^.},\alpha_t^u) \d t + \sigma(u,X_t^u, \P_{X_t^.},\alpha_t^u) \d W_t^u \quad 0 \leq t \leq T, \\
        X_0^u &= \xi^u, u \in I.
    \end{aligned}\right.
\end{align}
The initial condition is given by a collection $\xi=(\xi^u)_{u \in I}$ of $\R^d$-valued random variable satisfying 

\vspace{2mm}

- there exists a Borel measurable function $\xi : I \times (0,1) \to \R^d$ such that
\begin{align} \label{condIC} 
\xi^u ~=~ \xi(u,U^u)\mbox{ for } u \in I, & \quad \mbox{ and } \quad   \int_{I} \mathbb{E}\big[|\xi^u|^2 \big] \d u ~< ~\infty. 
\end{align} 



\vspace{2mm}

\noindent The control is given by a collection $\alpha=(\alpha^u)_{u \in I}$ of $A$-valued random processes satisfying 

\vspace{2mm}

- there exists a Borel measurable function $\alpha : I \times [0,T] \times \Cc^n_{[0,T]} \times (0,1) \rightarrow A$, such that
\begin{align}
    \alpha_t^u = \alpha(u,t,W^u_{. \wedge t},U^u), \quad t \in [0,T],~ u \in I\;, & \;  \mbox{ and } \; 
    \int_{I} \int_{0}^{T} \mathbb{E}\big[|\alpha_t^u|^2 \big] \d t \d u < + \infty, 
\end{align}
where $W^u_{. \wedge t}$ is the path $s \mapsto W^u_{s \wedge t}, s \in [0,T]$. 
Such a control $\alpha$ is called admissible and we denote by $\Ac$ the class of all admissible controls.

\vspace{2mm}

We next give the precise definition of a solution to the non exchangeable mean field controlled SDE.
\begin{Definition}\label{def : solution SDE}
Fix an initial condition $\xi = (\xi^u)_{u \in I}$ satisfying \eqref{condIC} and an admissible control  $\alpha\in \Ac$.
We define a solution to equation \eqref{eq : control process X} as a collection 
$X=(X^u)_{u\in I}\in L^2(I; \Sc^{d})$ such that 
\begin{enumerate}[(i)]
\item the map
$u\mapsto \Lc(X^{u}, W^u, U^u)$
is Borel measurable from $I$ to $\Pc_2(C_{[0,T]}^d\times C_{[0,T]}^n\times (0,1))$,
\item $X^u$ solves \eqref{eq : control process X} for all $u\in I$.
\end{enumerate}
We say that the solution is unique if, whenever 
$(X^u)_{u\in I}$ and $(\tilde X^u)_{u\in I}$ are solutions to 
  \eqref{eq : control process X}   then  the processes $X^u$ and $\tilde X^u$ coincide, up to a $\P$-null set,  for almost all $u\in I$.
\end{Definition}

From Theorem 2.1 in \cite{de2024mean} we have the following existence and uniqueness result for the controlled mean-field non exchangeable system. 
\begin{Theorem}\label{THMexistUniqX}
    Under Assumption \ref{assumptionzero}, for an initial condition $\xi = (\xi^u)_{u \in I}$ satisfying \eqref{condIC} and an admissible control  $\alpha\in \Ac$, there exists a unique solution $X$ to
 \eqref{eq : control process X} in the sense of Definition \ref{def : solution SDE}. Moreover,  there exists a Borel measurable function  $\mrx$ from $I \times [0,T] \times C^n_{[0,T]} \times (0,1)$ to $\R^d$ such that $X_t^u = \mrx(u,t,W^u_{. \wedge t}, U^u)$ $\P$-a.s. for all  $(t,u)\in [0,T]\times I$.
\end{Theorem}

\begin{proof}
   We fix $\xi$ and $\alpha \in \Ac$. The existence and uniqueness of $X=(X^u)_{u\in I}$ is given by Theorem 2.1 in \cite{de2024mean}. We therefore only need to prove the existence of a Borel measurable function  $x$ such that $X_t^u = \mrx(u,t,W^u_{. \wedge t}, U^u)$ $\P$-a.s. for all  $(t,u)\in [0,T]\times I$.
   To prove the result, we proceed in two steps.

\vspace{2mm}

\noindent   \textit{Step 1.}   
   For any $\nu$ in $L^2\big(\Pc_2(\Cc^d_{[0,T]})\big)$ and for any $u \in I$, we denote by $X^{\nu,u}$ the solution to 
   \begin{align}
       \left\{\begin{aligned}
        \d X_t^{\nu,u} &= b(u,X_t^{\nu,u},(\nu_t^v)_v,\alpha_t^u) \d t + \sigma(u,X_t^{\nu,u},(\nu_t^v)_v,\alpha_t^u) \d W_t^u \quad 0 \leq t \leq T, \\
        X_0^u &= \xi^u, u \in I.
    \end{aligned}\right.
   \end{align}
 Under Assumption \ref{assumptionzero}, the process $X^{\nu,u}$ is uniquely defined for any $\nu\in L^2 \big(\Pc_2(\R^d)\big)$ and $u\in I$.  
We prove that  for any $\nu \in L^2\big(\Pc_2(\Cc^d_{[0,T]})\big)$, the process $X^{\nu,u}$ admits the required property: there exists a Borel measurable function  $\mrx^\nu$ such that $X_t^{\nu,u} = \mrx^\nu(u,t,W^u_{. \wedge t}, U^u)$ $\P$-a.s. for all  $(t,u)\in [0,T]\times I$.

Define the sequence of processes $(X^{\nu,m,u})_{m\geq0}$ by  $X_t^{\nu,0,u}=\xi^u$ for $t\in[0,T]$ and 
\begin{align}
       \left\{\begin{aligned}
        \d X_t^{\nu,m+1,u} &= b(u,X_t^{\nu,m,u},(\nu_t^v)_v,\alpha_t^u) \d t + \sigma(u,X_t^{\nu,m,u},(\nu_t^v)_v,\alpha_t^u) \d W_t^u \quad 0 \leq t \leq T, \\
        X_0^u &= \xi^u, u \in I.
    \end{aligned}\right.
   \end{align}
   for $m\geq0$.

Using  Lemma 10.1 in  \cite{rogers_diffusions_2000}, 
we get by an induction argument that there exists a Borel measurable function  $\mrx^{\nu,m}$ such that $X_t^{\nu,m,u} = \mrx^{\nu,m}(u,t,W^u_{. \wedge t}, U^u)$ $\P$-a.s. for all  $(t,u)\in [0,T]\times I$ and all $m\geq0$.

From Assumption \ref{assumptionzero}, we have by classical estimates on diffusion processes
\begin{align}\label{convXnunXnu}
\sup_{u\in I}\E\Big[\sup_{t\in[0,T]}|X_t^{\nu,m,u}-X_t^{\nu,u}|^2\Big] \xrightarrow[m\rightarrow+\infty]{}0\;.
\end{align}
Define the measurable maps $\mrx^{\nu}$ by 
\begin{align*}
\mrx^{\nu} & =  \limsup_{m\rightarrow+\infty}~ \mrx^{\nu,m}.
\end{align*}
From \eqref{convXnunXnu}, we get $X_t^{\nu,u} = \mrx^{\nu}(u,t,W^u_{. \wedge t}, U^u)$ $\P$-a.s. for all  $(t,u)\in [0,T]\times I$. 

   \vspace{2mm}

\noindent   \textit{Step 2.} We turn to the general case. Define the sequence of processes $(X^{m,u})_{n\geq0}$ by  $X_t^{0,u}=\xi^u$ for $t\in[0,T]$ and 
\begin{align}
       \left\{\begin{aligned}
        \d X_t^{m+1,u} &= b(u,X_t^{m+1,u},\P_{X^{m,.}_t},\alpha_t^u) \d t + \sigma(u,X_t^{\nu,m+1,u},\P_{X^{m,.}_t},\alpha_t^u) \d W_t^u \quad 0 \leq t \leq T, \\
       X_0^u &= \xi^u, u \in I.
    \end{aligned}\right.
   \end{align}
   for $n\geq0$.
 From Step 1, we have by an induction argument $X_t^{m,u} = \mrx^{m}(u,t,W^u_{. \wedge t}, U^u)$ $\P$-a.s. for all  $(t,u)\in [0,T]\times I$ with $\mrx^n$ a Borel function  for all $n\geq 0$. 
We next define the measurable map $\mrx$ by 
\begin{align}\label{xlimsupxn}
\mrx & =  \limsup_{m\rightarrow+\infty}~\mrx^{m}
\end{align}
From Step II of the proof of Theorem 2.1 in \cite{de2024mean}, we have for $r$ large enough
\begin{align}\label{convXnkX}
\int_{I}\E\Big[\sup_{t\in[0,T]}|X_t^{mr,u}-X_t^{u}|^2\Big]du \xrightarrow[m\rightarrow+\infty]{}0\;.
\end{align}
From \eqref{xlimsupxn} and \eqref{convXnkX}, we get the result.  
\end{proof}

\subsection{The control problem}
We next introduce two reward functions 
\beqs 
f:~I\times \R^d\times L^2\big(\Pc_2(\R^d) \big)\times A\to \R & \mbox{ and }&
  g:I\times \R^d\times L^2 \big(\Pc_2(\R^d)\big)\to \R
\enqs
and we make the following assumption. 
\begin{Assumption}\label{assumptiononfgbis}$\,$
The functions $f$ and $g$
 are Borel measurable and there exists a constant   $M\ge 0$ such that 
 \begin{align}
- M \big(1+|x|^2+  \bd(\mu,\delta_0)^2  \big)\leq  f(u,x,\mu,a)\leq  M \big(1+|x|^2+  \bd(\mu,\delta_0)^2 + |a|^2 \big) , 
\end{align} 
 and 
 \begin{align}
 |g(u,x,\mu)| \leq  M  \big( 1+ |x|^2 + \bd(\mu,\delta_0)^2 \big),
\end{align} 
 for every  $u\in I$, $a\in A$, $x\in\R^d$, $\mu\in L^2 (\Pc_2(\R^d))$. 
 \end{Assumption}

We define the cost functional $J:~\Ac\rightarrow\R$ by
\begin{align}
J(\alpha) &:= \; \int_{I} \mathbb{E} \big[\int_{0}^{T} f(u,X_t^u,\P_{X_t^\cdot},\alpha_t^u) dt  + g(u,X_T^u,\P_{X_T^\cdot}) \big] \d u.
\end{align}
Under Assumptions \ref{assumptionzero} and \ref{assumptiononfgbis}, we get from Theorem \ref{THMexistUniqX} that $J(\alpha)$ is well defined for any $\alpha\in\Ac$. Our aim is to study the optimal control problem consisting in minimizing the function $J$ over $\Ac$, that is, computing 
\begin{align} \label{eq : Fonctionnelle Coût}
V_0 & := \;  \inf_{\alpha\in\Ac} J(\alpha),
\end{align}
and finding an optimal control $\alpha^*\in\Ac$, that is
\begin{align}
J(\alpha^*) & = \;  V_0.
\end{align}
We notice that the bounds of $f$ in Assumption \ref{assumptiononfgbis} ensures that $V_0$ is finite.



\section{Hamiltonian and adjoint equations for non exchangeable SDE}

In this section, we present the key concepts that will serve as the foundation for analyzing the stochastic optimization problem \eqref{eq : Fonctionnelle Coût}.

\subsection{Derivative and convexity in $L^2(\Pc_2(\R^d))$}


 We first recall the definition of the linear functional derivative on $L^2 (\Pc_2(\R^d))$ introduced in \cite{de2024mean} as an extension of the derivative on the Wasserstein space $\Pc_2(\R^d)$.

\vspace{2mm}
 
For $\mu \in L_2 \big(\Pc_2(\R^d)\big)$ and $\phi:~I\times \R^d\rightarrow \R$ a measurable function  with quadratic growth in $x$, uniformly in $u$, we define the duality product $\langle \phi, \mu \rangle$ by
\begin{equation}\label{eq : duality product}
    \langle \phi, \mu \rangle := \int_{I} \int_{\R^d} \phi(u,x) \mu^u(\d x) \d u.
\end{equation}

\begin{Definition}
\label{Defderivative} (i) Given a  function
$v: L^2(\Pc_2(\R^d))\to \R$, we say that a measurable function 
\begin{align}
\dmu v:  L^2 (\Pc_2(\R^d)) \times I \times \R^d \; \ni \; (\mu,u,x)   \; \longmapsto \; \dmu v(\mu)(u,x) \in \R    
\end{align}
is the linear functional derivative (or flat derivative) of $v$ if
\begin{enumerate}
\item  $(\mu,x)\mapsto\dmu v(\mu)(u,x)$ is continuous from $L^2 (\Pc_2(\R^d))\times\R^d$ to $\R$ for  all $u\in I$;
    \item for every compact set $K\subset  L^2 (\Pc_2(\R^d))$ there exists a constant $C_K>0$ such that
    $$\left|
    \dmu v(\mu)(u,x) \right|\le C_K\, (1+|x|^2),  $$
    for all $u\in I$, $x\in \R^d$, $\mu\in K$;
    \item  we have
\beqs
v(\nu)-v(\mu) & = & \int_0^1 \langle \dmu v(\mu + \theta(\nu-\mu)),\nu - \mu \rangle \d \theta \\
 & =  &   \int_0^1 \int_I  \int_{\R^d} \dmu v(\mu + \theta(\nu-\mu)(u,x) \,(\nu^u-\mu^u)(\d x) 
 \d u \d \theta
\enqs
for all $\mu,\nu\in  L^2 (\Pc_2(\R^d))$.
\end{enumerate}
(ii) We say that the function $v$ admits a continuously differentiable flat derivative if 
\begin{enumerate}
\item $v$ admits a flat derivative $ \dmu v$ satisfying $x\mapsto  \dmu v(\mu)(u,x)$ is Fr\'echet differentiable with Fr\'echet derivative denoted by $x\mapsto\partial \dmu v(\mu)(u,x)$ for all $(\mu, u)\in L^2 (\Pc_2(\R^d))\times I$; 
\item  $(\mu,x)\mapsto\partial \dmu v(\mu)(u,x)$ is continuous from $L^2(\Pc_2(\R^d))\times\R^d$ to $\R$ for  all $u\in I$;
    \item for every compact set $K\subset  L^2 (\Pc_2(\R^d))$ there exists a constant $C_K>0$ such that
    $$\left|
    \partial \dmu v(\mu)(u,x) \right|\le C_K\, (1+|x|),  $$
    for all $u\in I$, $x\in \R^d$, $\mu\in K$.

\end{enumerate}
\end{Definition}

\begin{Remark}\label{Extended functions}
    In the sequel, the function $v$ will be mostly defined on $I \times \R^d \times L^2(\Pc_2(\R^d)) \to \R$. In this case, the flat derivative of $v$ is defined as a measurable function 
    \begin{equation}
        \dmu v: I \times \R^d \times  L^2(\Pc_2(\R^d)) \times I \times \R^d \; \ni \; (u,x,\mu,\tilde u,\tilde x)   \; \longmapsto \; \dmu v(u,x,\mu)(\tilde u,\tilde x) \in \R
        \end{equation}
        satisfying :
    \begin{enumerate}
    \item  $(x,\mu,\tilde{x}) \mapsto \frac{\delta}{\delta m} v(u,x,\mu)(\tilde u,\tilde{x})$ is continuous from   $\R^d \times L^2 (\Pc_2(\R^d)) \times \R^d$  to $\R$ for all $u,\tilde u\in I$;

    \item     For every compact set $K \subset L^2(\Pc_2(\R^d) )$, there exists a constant $C_K > 0$ such that :

    \begin{align}
        \Big | \frac{\delta}{\delta m}v(u,x,\mu)(\tilde{u},\tilde{x}) \Big |  \leq C_K \big( 1+ |x|^2 + | \tilde{x}|^2 \big)
    \end{align}
    for any $u,\tilde{u} \in I$, $x, \tilde{x} \in \R^d$ and $\mu \in K$;

    \item    we have 
    \begin{align}
    v(u,x,\nu) - v(u,x,\mu) &=  \int_0^1  \langle \frac{\delta}{\delta m}v \big(u,x,\mu + \theta(\nu - \mu)\big), \nu - \mu) \d \theta \notag \\
    &= \int_{0}^1 \int_I \int_{\R^d} \frac{\delta}{\delta m}v\big(u,x,\mu + \theta( \nu  - \mu) \big)(\tilde{u},\tilde{x})(\nu^{\tilde{u}} - \mu^{\tilde{u}})(\d \tilde{x}) \d \tilde{u} \d \theta
    \end{align}
    for any $u \in I, x \in \R^d$ and $\mu,\nu \in L^2 (\Pc_2(\R^d))$.
\end{enumerate}
In this case, the flat derivative is said to be Fr\'echet continuously differentiable if
\begin{enumerate}
\item $v$ admits a flat derivative $ \dmu v$ satisfying $\tilde x\mapsto  \dmu v(u,x,\mu)(\tilde u,\tilde x)$ is Fr\'echet differentiable with Fr\'echet derivative denoted by $\tilde x\mapsto\partial \dmu v(u,x,\mu)(\tilde u,\tilde x)$ for all $(x,\mu, u,\tilde u)\in \R^d\times L^2 (\Pc_2(\R^d))\times I^2$;
\item  $(\mu,x, \tilde x)\mapsto\partial \dmu v(u,x,\mu)(\tilde u,\tilde x)$ is continuous from $L^2(\Pc_2(\R^d))\times\R^d\times\R^d$ to $\R^d$ for  all $u,\tilde{u} \in I$;
    \item for every compact set $K\subset  L^2(\Pc_2(\R^d))$ there exists a constant $C_K>0$ such that
    $$\left|
    \partial \dmu v(u,x,\mu)(\tilde u,\tilde x) \right|\le C_K\, (1+|x|+|\tilde x|),  $$
    for all $u, \tilde  u\in I$, $x, \tilde  x\in \R^d$, $\mu\in K$.\end{enumerate}
\end{Remark}

We next give some examples of functions for which we compute the linear functional derivatives.

\begin{Example}\label{example linear functional derivative}
\begin{enumerate}
\item[(i)] {\it Linear functions}
$$v(\mu)=
\int_I\int_{\R^d} \varphi(u,x)
\mu^u(\d x)\d u,
$$
where $\varphi$ is a measurable function with quadratic growth in $x$. 
Then, 
\begin{align}
\dmu v(\mu)(u,x) &= \varphi(u,x).     
\end{align}


\item[(ii)] Collection of cylindrical functions: Let $\phi= (\phi_1,\ldots,\phi_k)$ be an $\R^k$ valued  function defined on $I \times \R^d$.
\begin{align}
v(\mu) &= \; \int_I     
F\Big(\int_{\R^d}\phi(u,x) \mu^u(dx) \Big) \d u. 
\end{align}
Denoting the partial derivatives of $F$ by $\partial_iF$, we have, 
\begin{align*}
\dmu v(\mu)(u,x)  &=  \sum_{i=1}^k \partial_iF
\Big(\int_{\R^d}
\phi(u,x) \mu^u(dx) \Big) \,\phi_i(u,x).
\end{align*}
\item[(iii)]\label{example : Cylindrical functions} Cylindrical functions of measure collection: 
 Let $\phi= (\phi_1,\ldots,\phi_k)$ be an $\R^k$ valued  function defined on $I \times \R^d$.
\begin{align} 
v(\mu) & = \; F\Big(
\int_I\int_{\R^d}\phi(u,x) \,\mu^u(\d x)\d u
\Big),
\end{align} 
where $F:\R^k\to \R$ and $\phi_i$ are real functions on $\R^d$. 
Then, 
\begin{align*}
\dmu v(\mu)(u,x) &= \sum_{i=1}^k \partial_iF
\Big(
\int_I\int_{\R^d}\phi(u,x) \,\mu^u(\d x)\d u
\Big)\phi_i(u,x)
\end{align*}
\item[(iv)] $k$-interaction functions:
\begin{align}
v(\mu) &= \; \int_{I^k} 
\int_{(\R^d)^k} \varphi(u_1,\ldots,u_k,x_1,\ldots,x_k) 
\mu^{u_1}(\d x_1) \ldots \mu^{u_k}(\d x_k) \d u_1\ldots\d u_k.      
\end{align}
Then, 
\begin{align}
\dmu v(\mu)(u,x) = \sum_{i=1}^k \int_{U^{k-1}} \int_{(\R^d)^{k-1}} 
\varphi(u_1,\ldots,u_{i-1},u,u_{i+1},\ldots,u_k,x_1,\ldots,x_{i-1},x,x_{i+1},\ldots,x_k) \\
  \mu^{u_1}(\d x_1) \ldots  \mu^{u_{1-1}}(\d x_{i-1}) 
 \mu^{u_{1+1}}(\d x_{i+1}) \ldots \mu^{u_k}(\d x_k) \d u_1 \ldots 
 \d u_{i-1} \d u_{i+1} \ldots \d u_k.  
\end{align}

\end{enumerate}
\end{Example}

We now present a key lemma that will play a crucial role in the next section, in the derivation of the Pontryagin maximum principle.

\begin{Lemma}\label{lemma 1bis}
Let $f: L^2(\Pc_2(\mathbb{R}^d)) \to \mathbb{R}$. Suppose that $f$  has a continuously differentiable  linear functional derivative $\dmu{f}$ . 
 For $X,Y \in L^2 \big(\Omega,\Fc,\P,\R^d\big)^I$ such that $\P_{X^.},\P_{Y^.}\in L^2(\Pc_2(\mathbb{R}^d))$ we have
\begin{equation}\label{eq : lemma1}
     \underset{\epsilon \to 0}{\lim} \frac{1}{\epsilon} \Big(f(\P_{X^.+ \epsilon Y^.}) - f(\P_{X^.})\Big) = 
     \int_{I} \mathbb{E}\big[\partial_{} \frac{\delta f}{\delta m}(\P_{X^.})(u,X^u) \cdot Y^u \big] \d u.  \quad 
\end{equation}
\end{Lemma}

\begin{proof}
    By definition of  the linear functional derivative $\dmu f$ and its Fr\'echet derivative $\partial\dmu f$, we have :
     \begin{align}
       f(\P_{X^.+ \epsilon Y^.}) - f(\P_{X^.}) &= \int_{0}^{1} \int_{I} \int_{\mathbb{R}^{d}} \frac{\delta }{\delta m}f (  \theta \P_{X^.+ \epsilon Y^.} + (1-\theta) \P_{X^.})(u,x) (\P_{X^u + \epsilon Y^u} - \P_{X^u})(dx) \d u  \d \theta \notag \\
       &= \int_{0}^{1} \int_{I} \mathbb{E}\Big[\frac{\delta }{\delta m}f(\theta \P_{X^.+ \epsilon Y^.}+ (1-\theta) \P_{X^.})(u,X^u + \epsilon Y^u) \notag \\
       & \hspace{0.35 cm}-  \frac{\delta }{\delta m}f(\theta  \P_{X^.+ \epsilon Y^.} + (1-\theta) \P_{X^.})(u,X^u) \Big] \d u \d \theta \notag \\
       &= \epsilon  \int_{0}^{1} \left( \int_{I} \mathbb{E} \Big[\int_{0}^{1} \partial\frac{\delta }{\delta m}f(\theta \P_{X^.+ \epsilon Y^.} + (1-\theta) \P_{X^.})(u,X^u + \lambda \epsilon Y^u) \cdot Y^u   \d \lambda \Big] \d u \right) \d \theta \notag\;.
   \end{align}
We then notice that 
   \begin{align}
\bd(\P_{X^.+ \epsilon Y^.},\P_{X^.}) \leq \epsilon \left( \int_{I} \E \big[|Y^u|^2 \big]  \d u\right)^{\frac{1}{2}}\xrightarrow[\epsilon \rightarrow 0]{}0\;.
   \end{align}
 Moreover, the set $K=\lbrace \theta \P_{X^.+ \epsilon Y^.} + (1-\theta) \P_{X^.}: \epsilon \in [0,1], \theta \in [0,1]\rbrace$ is compact as the image of the compact set $[0,1]^2$ by the continuous map $(\theta,\eps)\mapsto \theta \P_{X^.+ \epsilon Y^.} + (1-\theta) \P_{X^.}$. 
 From Cauchy-Schwarz inequality, the quadratic growth and continuity of  $\partial \frac{\delta }{\delta m}f$  we conclude from the convergence dominate theorem.
\end{proof}

\begin{Remark}
In the case where $f$ is defined on $I \times \R^d \times L^2(\Pc_2(\R^d)) \times A$ and has a continuously differentiable linear functional derivative $\frac{\partial}{\partial m} f$ in the sense of Remark \ref{Extended functions}, Lemma \ref{lemma 1bis} gives  
\begin{align}
        \underset{\epsilon \to 0}{\lim} \frac{1}{\epsilon} \big( f(u,x,\P_{X^.+ \epsilon Y^.},a) - f(u,x,\P_{X^.},a) \big)=  \int_{I} \mathbb{E}\big[\partial \frac{\delta }{\delta m}f(u,x,\P_{X^.},a)(\tilde{u},X^{\tilde{u}}) \cdot Y^{\tilde{u}} \big] \d \tilde{u},
\end{align}
 for $(u,x,a)\in I\times \R^d\times A$ and $X,Y \in L^2 \big(\Omega,\Fc,\P,\R^d\big)^I$ such that $\P_{X^.},\P_{Y^.}\in L^2(\Pc_2(\R^d))$.
\end{Remark}

We now introduce a concept of convexity that aligns with the notion of differentiability previously defined. A  function $f$ defined on $L^2(\Pc_2(\R^d))$ assumed to have a continuous differentiable flat derivative in the sense defined above is said to be convex if for every $\mu=(\mu^u)_{u}, \mu'=((\mu')^u)_{u} \in L^2(\Pc_2(\R^d))$, we have 

\begin{equation}\label{eq : convex case 1}
    f(\mu')-f(\mu) \geq \int_{I} \mathbb{E} \big[\partial \frac{\delta }{\delta m}f(\mu)(u,X^u) \cdot (X'^{u} - X^{u}) \big] \d u,
\end{equation}
where $X^{u} \sim \mu^{u}$ and $X'^{u} \sim (\mu')^{u}$.

\noindent More generally,  a function $f$ defined on $I \times \R^d \times L^2 (\Pc_2(\R^d))$ which is jointly differentiable in $(x,\mu) \in \R^d \times L^2 (\Pc_2(\R^d)) $ is said to be convex if for every $(x,\mu)$ and $(x',\mu')$ and for every $u \in I$, we have 

\begin{equation}\label{eq : convex case 2}
    f(u,x',\mu') - f(u,x,\mu) \geq \partial_x f(u,x,\mu) \cdot (x'-x) + \int_{I} \mathbb{E} \big[\partial \frac{\delta }{\delta m}f(u,x,\mu)(\tilde{u},X^{\tilde{u}}) \cdot (X'^{\tilde{u}}-X^{\tilde{u}}) \big]  \d \tilde{u}.
\end{equation}

\begin{Remark}\label{remark: convex case}
    In the most general setting, the function $f$ defined on $I \times \R^d \times L^2 \big(\Pc_2(\R^d)\big) \times A$ assumed to be jointly differentiable in $(x,\mu,a) \in \R^d \times L^2 \big(\Pc_2(\R^d)\big) \times A$ is said to be convex if for every $(x,\mu,a)$ and $(x',\mu',a')$ and for every $u \in I$, we have 
    \begin{align}\label{eq : convexity function}
        f(u,x',\mu',a') - f(u,x,\mu,a) &\geq \partial_x f(u,x,\mu,a) \cdot (x'-x)  + \partial_{a} f(u,x,\mu,a) \cdot (a'-a) \\
         &\quad + \int_{I} \mathbb{E} \big[\partial \frac{\delta }{\delta m}f(u,x,\mu,a)(\tilde{u},X^{\tilde{u}}) \cdot (X'^{\tilde{u}}-X^{\tilde{u}}) \big] \d \tilde{u}.
    \end{align}

\end{Remark}

\subsection{The Hamiltonian and the dual equations}

In our framework, the Hamiltonian of the stochastic optimization problem is defined as the function $H$ valued in $\R$ given by 

\begin{equation}\label{eq : Hamiltonien}
    H(u,x,\mu,y,z,a) := b(u,x,\mu,a) \cdot y + \sigma(u,x,\mu,a):z + f(u,x,\mu,a),
\end{equation}
for $(u,x,\mu,y,z,a) \in I \times \R^d \times L^2 (\Pc_2(\R^d)) \times \R^d \times \R^{d \times n} \times A$.

\begin{Definition}\label{def : adjoint processes}
Suppose  that the coefficients $b$, $\sigma$, $f$ and $g$ are differentiable with respect to $x$ and admit differentiable  flat derivatives with respect to $\mu$. Fix an admissible control $\alpha \in \Ac$ and inital condition $\xi$, and  denote by $X= (X^{u})_{u\in I}$ the solution to \eqref{eq : control process X}. 
Suppose also that 
\begin{equation}
    \int_{I} \E\bigg[\int_{0}^{T} \left[ | \partial_x f(u,X_t^u,\P_{X_t^.} , \alpha_t^u)|^2 + \int_I \tilde{\E} \big[ |\partial \frac{\delta }{\delta m}f(\tilde{u},\tilde{X}_t^u,\P_{X_t^.} , \tilde{\alpha}_t^u)(u,X_t^u) |^2  \big] \d \tilde{u} \right] \d t\bigg] \d u  < + \infty ,
\end{equation}
and
\begin{equation}
    \int_{I} \E \bigg[ | \partial_x g(u,X_T^u,\P_{X_T^.})|^2  + \int_{I} \tilde{\E} \big[| \partial \frac{\delta }{\delta m}g(\tilde{u},\tilde{X}_T^u,\P_{X_T^.})(u,X_T^u)|^2 \big] \d \tilde{u} \bigg] \d u < + \infty, 
\end{equation}
where $(\tilde X^u,\tilde \alpha^u)$ is a copy of $(X^u,\alpha^u)$ defined on another probability space $(\tilde{\Omega},\tilde{\Fc},\tilde{\P})$ with expectation operator $\tilde \E$.

We call adjoint processes of $X$ any pair $(Y,Z) = (Y_t^u,Z_t^u)_{u \in I, 0 \leq t \leq T}$ of processes in  $L^2(I; \Sc^{d}) \times L^2(I; \Hc^{2,d \times n})$  satisfying the following conditions.
\begin{enumerate}[(i)]
\item $(Y,Z)$ is solution to the adjoint equations 
 \begin{align}  \label{eq : adjoint equations}
        \begin{cases} 
        dY_t^u  = - \partial_{x} H(u,X_t^u,\P_{X_t^.},Y_t^u,Z_t^u,\alpha_t^u) \d t  + Z_t^u \d W_t^u  \\
         \qquad  \quad - \int_{I} \mathbb{\tilde{E}}\left[ \partial \frac{\delta }{\delta m}H(\tilde{u},\tilde{X}_t^{\tilde{u}},\P_{X_t^.},\tilde{Y}_t^{\tilde{u}},\tilde{Z}_t^{\tilde{u}},\tilde{\alpha}_t^{\tilde{u}})(u,X_t^u)\right] \d \tilde{u} \d t \;,\quad t\in[0,T]\;, \\
        \; Y_T^u   =  \partial_{x} g(u,X_T^u, \P_{X_T^.}) + \int_{I}  \mathbb{\tilde{E}} \left[\partial \frac{\delta }{\delta m}g(\tilde{u},\tilde{X}_T^{\tilde{u}},\P_{X_T^.})(u,X_T^{u}) \right] \d \tilde{u}\;, \\
        \end{cases}
\end{align}
for  every $u \in I$ where $(\tilde{X},\tilde{Y},\tilde{Z},\tilde{\alpha})$ is an independent copy of $(X,Y,Z,\alpha)$ defined on $(\tilde{\Omega},\tilde{\Fc},\tilde{\P})$. 
\item There exists Borel  functions  $\mry$ and $\mrz$ defined on $I \times [0,T]  \times \Cc^n_{[0,T]} \times (0,1)$ such that
\begin{align}
    Y_t^u = \mry(u,t,W^u_{. \wedge t}, U^u), \quad \text{ and } \quad  Z_t^u = \mrz(u,t,W^u_{. \wedge t}, U^u).
\end{align}
for $t\in[0,T]$, $\P$-a.s. and $u\in I$.
\end{enumerate}
\end{Definition}

\section{Pontryagin principle for optimality}

\noindent In this section, we explore the necessary and sufficient conditions for optimality under suitable convexity assumptions on the Hamiltonian. Before that, we introduce the key regularity properties that will be used throughout this section and the reminder of the paper.

\begin{Assumption}\label{assumptionbasic}$\,$ 
\begin{enumerate}[(i)]
\item The functions $b$ and $\sigma$
 are   differentiable with respect to $(x,a)$, $\partial_x (b,\sigma)$ and $\partial_{a} (b,\sigma)$ are  bounded  and the mappings $(x,\mu,a) \mapsto \partial_x(b,\sigma,f)(u,x,\mu,a)$ and $(x,\mu,a) \mapsto \partial_{a}(b,\sigma,f)(u,x,\mu,a)$ are continuous for any $u \in I$.
\item The functions $b$ and $\sigma$ are assumed to have Fr\'echet differentiable linear functional derivatives $\partial \dmu{b}$ and $\partial \dmu{\sigma}$ satisfying the following  property: there exist  constants $L$ and $M$ such that
 \begin{align}
    \big|\partial \dmu{b}(u,x,\mu,a)(\tilde u,\tilde x)-\partial \dmu{b}(u,x',\mu',a)(\tilde u,\tilde x')\big| &~ \leq ~L(|x-x'|+|\tilde x-\tilde x'|+\bd(\mu,\mu')),\\
        \big|\partial \dmu{\sigma}(u,x,\mu,a)(\tilde u,\tilde x)-\partial \dmu{\sigma}(u,x',\mu',a)(\tilde u,\tilde x')\big| &~ \leq ~L(|x-x'|+|\tilde x-\tilde x'|+\bd(\mu,\mu')),
 \end{align}
for all $u,\tilde u\in I$, $x,x',\tilde x,\tilde x'\in \R^d$ and $\mu,\mu'\in L^2(\Pc_2(\R^d))$ and
 \begin{align}
\big|\partial \dmu{b}(u,0,\delta_0,a)(\tilde u,0)\big|+
\big|\partial \dmu{\sigma}(u,0,\delta_0,a)(\tilde u,0)\big|
 \le M\big( 1+|a|\big),
  \end{align}
 for all  $u,\tilde u \in I$ and $a\in A$.
\item The functions $f$ (resp. $g$) is differentiable with respect to $(x,a)$ (resp. $x$), $\partial_x f$ and $\partial_{a} f$ (resp. $\partial_x g$) are locally bounded uniformly in $u\in I$ and the mappings $(x,\mu,a) \mapsto \partial_x f(u,x,\mu,a)$ and $(x,\mu,a) \mapsto \partial_{a}f (u,x,\mu,a)$  (resp. $\partial_x g(u,x,\mu)$) are continuous for any $u \in I$.
\item The functions $f$ and $g$ admit Fr\'echet differentiable linear functional derivatives. Moreover, for any $X \in L^2(I;\Sc^d) $ the following quantities are locally bounded uniformly in $u \in I$
 \begin{align}
     \int_{I} \tilde{\E}\big[ |\partial \frac{\delta }{\delta m}f(u,x,\mu,\alpha)(\tilde{u},X^{\tilde{u}})|^2 \big] \d \tilde{u}, \quad \text{ and }  \int_{I} \tilde{\E}\big[ |\partial \frac{\delta }{\delta m}g(u,x,\mu)(\tilde{u},X^{\tilde{u}})|^2 \big] \d \tilde{u}.
 \end{align}

\end{enumerate}  
\end{Assumption}

\subsection{A necessary condition}

We denote for $\alpha \in \Ac$,   $X = X^{\alpha}$ the collection of stochastics processes $(X^{\alpha,u}_t)_{u \in I,t \in [0,T]}$ solution of \eqref{eq : control process X} with admissible initial condition $\xi $. We now choose $\beta \in L^2(I;\Hc^{2,m})$ such that $\alpha + \epsilon (\beta- \alpha) \in \Ac$ for $0 \leq \epsilon \leq 1$ as $A$ is convex. We denote by $\delta: = \beta - \alpha$.

 In the following, we note $ \big(\theta_t^u = (X_t^u, \mathbb{P}_{X_t^.},\alpha_t^u) \big)_{0 \leq t \leq T, u \in I}$ and we define the variation processes $V=(V^u)_{u \in I}$ as the solution of the following stochastic differential equation 

\begin{align}\label{eq : SDE V}
    dV_t^u &= \left[\gamma_t^uV_t^u + \int_{I} \mathbb{\tilde{E}}\left[\partial  \frac{\delta }{\delta m}b(u,\theta_t^u)(\tilde{u},\tilde{X}_t^{\tilde{u}})\tilde{V}_t^{\tilde{u}} \right] \d \tilde{u}  +  \eta_t^u\delta_t^u \right] \d t   \\
    &\quad + \left[\hat{\gamma}_t^uV_t^u + \int_{I} \mathbb{\tilde{E}}\left[\partial  \frac{\delta}{\delta m}\sigma(u,\theta_t^u)(\tilde{u},\tilde{X}_t^{\tilde{u}})\tilde{V}_t^{\tilde{u}}\right] \d \tilde{u}  + \hat{\eta}_t^u  \delta_t^u \right]\d W_t^u, \notag
\end{align}
valued in $\R^d$ with $V_0^u=0$ and 
\begin{enumerate}
    \item [$\bullet$] $\gamma_t^u = \partial_{x} b\left(u,\theta_t^u \right) \in \R^{d \times d}$,
    \item [$\bullet$] $\eta_t^u = \partial_{a} b \left(u,\theta_t^u \right) \in \R^{d \times m}$,
    \item [$\bullet$] $\hat{\gamma}_t^u = \partial_{x} \sigma(u,\theta_t^u)$ $\in \mathbb{R}^{(d \times n) \times d}$,
    \item [$\bullet$] $\hat{\eta}_t^u = \partial_{a} \sigma(u,\theta_t^u)$ $\in \mathbb{R}^{(d \times n) \times m }$,
\end{enumerate}
for $t\in[0,T]$ and $u\in I$.
The process $(X,V)$ satisfies a coupled SDE falling into the framework of \eqref{eq : control process X} with control pair $(\alpha,\beta)$ and admissible initial condition $(\xi,0)$. Under Assumptions \ref{assumptionzero} and  \ref{assumptionbasic} (i)-(ii), we get from Theorem \ref{eq : control process X} existence and uniqueness of $(X,V)$ and hence $V$.

\begin{Lemma}\label{lemDerX=V}
     For $\epsilon  \in [0,1]$, we define the admissible control  $\alpha^{\epsilon}:= \alpha + \epsilon(\beta - \alpha)= \alpha + \epsilon \delta$ and by $X^{\epsilon}=(X^{\epsilon,u})_{u} =(X^{\alpha^{\epsilon},u})_u$ the corresponding controlled state.  Then, the mapping $u \mapsto \P_{(X^u,X^{\epsilon,u},V^u)}$is measurable  and 
    \begin{align}\label{derXV}
        \underset{\epsilon \to 0}{\lim}  \int_{I}  \mathbb{E} \Big[\underset{ 0 \leq t \leq T}{\sup}  \big| \frac{X_t^{\epsilon,u} - X_t^{u}}{\epsilon} - V_t^{u} \big|^2 \Big] \d u ~ = ~0.
    \end{align} 
\end{Lemma}

\begin{proof}
Let $\epsilon > 0$. The mesurability of $u \mapsto \P_{(X^u,X^{\epsilon,u},V^u)}$ is a consequence of the representation of $X^u$ and $V^u$ and $X^{\epsilon,u}$ by measurable mappings $\mrx$, $\mrx^{\epsilon}$ and $\mrv$  such that 

\begin{align}
    X_t^u = \mrx(u,t,W^u_{. \wedge t}, U^u), \quad  X_t^{\epsilon,u} = \mrx^{\epsilon}(u,t,W^u_{. \wedge t},U^u), \quad V_t^u = \mrv(u,t,W^u_{. \wedge t}, U^u),
\end{align}
    We now prove \reff{derXV}. We define $\theta_t^{\epsilon,u} = (X_t^{\epsilon,u}, \P_{X_t^{\epsilon,.}},\alpha_t^{\epsilon,u})$ and $V_t^{\epsilon,u} = \frac{1}{\epsilon}(X_t^{\epsilon,u} - X_t^u)- V_t^u$ for  $t\in[0 ; T]$, and  $u \in I$. From the dynamics of $X$, $X^\eps$ and $V$ we have
\begin{align}\label{dynVeps}
    dV_t^{\epsilon,u}
    &=V_t^{\epsilon,u,1} \d t + V_t^{\epsilon,u,2} \d W_t^u, \notag
\end{align}
with
\begin{align}
V_t^{\epsilon,u,1} &~=~\frac{1}{\epsilon}\big[b(u,\theta_t^{\epsilon,u}) - b(u,\theta_t^u) \big] - \partial_{x} b(u,\theta_t^u)V_t^u - \partial_{a} b(u,\theta_t^u)\delta_t^u -\int_{I} \mathbb{\tilde{E}}\left[\partial  \frac{\delta }{\delta m}b(u,\theta_t^u)(\tilde{u},\tilde{X}_t^{\tilde{u}})\tilde{V}_t^{\tilde{u}} \right] \d \tilde{u}  \;, \\
V_t^{\epsilon,u,2}    &~=~\frac{1}{\epsilon}\big[\sigma(u,\theta_t^{\epsilon,u}) - \sigma(u,\theta_t^u)\big] - \partial_{x} \sigma(u,\theta_t^u)V_t^u - \partial_{a} \sigma(u,\theta_t^u)\delta_t^u - \int_{I} \mathbb{\tilde{E}}\left[\partial  \frac{\delta }{\delta m}\sigma(u,\theta_t^u)(\tilde{u},\tilde{X}_t^{\tilde{u}})\tilde{V}_t^{\tilde{u}}\right] \d \tilde{u}  \;.
    \end{align}
We next define 
\begin{align}
    X_t^{\epsilon,\lambda,u} = X_t^u + \lambda \epsilon (V_t^{\epsilon,u} + V_t^{u})\;,\quad \alpha_t^{\epsilon,\lambda,u} = \alpha_t^{u} + \lambda \epsilon \delta_t^u~~\mbox{ and }~~\theta_t^{\epsilon,\lambda,u}=(X_t^{\epsilon,\lambda,u},\P_{X_t^{\epsilon,\lambda,.}},\alpha_t^{\epsilon,\lambda,u}),
\end{align}
for $t \in [0,T]$, $u \in I$ and $\lambda \in [0,1]$.
From Lemma \ref{eq : lemma1}, we have 
\begin{align}
    \frac{1}{\epsilon}\left[b(u,\theta_t^{\epsilon,u}) - b(u,\theta_t^{u}) \right] &= \int_{0}^{1} \partial_{x}b(u,\theta_t^{\epsilon,\lambda,u})(V_t^{\epsilon,u} + V_t^u) \d \lambda + \int_{0}^{1}  \partial_{a}b(u,\theta_t^{\epsilon,\lambda,u})\delta_t^u \d \lambda \notag \\
    &\quad + \int_{0}^{1} \int_{I} \mathbb{\tilde{E}}\left[\partial  \frac{\delta }{\delta m}b(u,\P_{X_t^{\epsilon,\lambda,.}},{X}_t^{\epsilon,\lambda,u},\alpha_t^{\epsilon,\lambda,u})(\tilde{u},\tilde{X}_t^{\epsilon,\lambda,\tilde{u}})(\tilde{V}_t^u+ \tilde{V}_t^{\epsilon,u}) \right] \d \tilde{u}  \d \lambda.
\end{align}
Plugging this expression in $V_t^{\epsilon,u,1}$ gives
\begin{align}
    V_t^{\epsilon,u,1}&= \int_{0}^{1} \partial_x b(u,\theta_t^{\epsilon,\lambda,u}) V_t^{\epsilon,u} d\lambda + \int_{0}^{1} \int_{I} \mathbb{\tilde{E}}\left[\partial  \frac{\delta }{\delta m}b(u,\P_{X_t^{\epsilon,\lambda,.}},{X}_t^{\epsilon,\lambda,u},\alpha_t^{\epsilon,\lambda,u})(\tilde{u},\tilde{X}_t^{\epsilon,\lambda,\tilde{u}})\tilde{V_t}^{\epsilon,u} \right] \d \tilde{u}  \d \lambda 
    \notag \\
    &\quad+ \int_{0}^{1} \left[\partial_{a}b(u,\theta_t^{\epsilon,\lambda,u}) - \partial_{a}b(u,\theta_t^{u})\right] \delta_t^u \d\lambda +  \int_{0}^{1}  \left[\partial_{x}b(u,\theta_t^{\epsilon,\lambda,u})-\partial_{x} b(u,\theta_t^{u})\right] V_t^u  \d \lambda \notag \\
    &\quad+ \int_{0}^{1} \int_{I} \mathbb{\tilde{E}}\bigg[\Big(\partial  \frac{\delta }{\delta m}b(u,\P_{X_t^{\epsilon,\lambda,.}},{X}_t^{\epsilon,\lambda,u},\alpha_t^{\epsilon,\lambda,u})(\tilde{u},\tilde{X}_t^{\epsilon,\lambda,\tilde{u}})-\partial  \frac{\delta }{\delta m}b(u,\P_{X_t^{.}},X_t^{u},\alpha_t^{u})(\tilde{u},\tilde{X}_t^{\tilde{u}})\Big.\tilde{V}_t^{\tilde{u}} \bigg] \d \tilde{u}  \d \lambda  \notag \\
    &= \int_{0}^{1} \partial_x b(u,\theta_t^{\epsilon,\lambda,u}) V_t^{\epsilon,u} d\lambda + \int_{0}^{1} \int_{I} \mathbb{\tilde{E}}\left[\partial  \frac{\delta }{\delta m}b(u,\P_{X_t^{\epsilon,\lambda,.}},{X}_t^{\epsilon,\lambda,u},\alpha_t^{\epsilon,\lambda,u})(\tilde{u},\tilde{X}_t^{\epsilon,\lambda,\tilde{u}})\tilde{V_t^{\epsilon,u}} \right] \d \tilde{u}  \d \lambda  \notag \\
    &\quad + I_t^{\epsilon,u,1} + I_t^{\epsilon,u,2} + I_t^{\epsilon,u,3},
\end{align}
where
\begin{align}
    I_t^{\epsilon,u,1} = & \int_{0}^{1} \big[\partial_{a}b(u,\theta_t^{\epsilon,\lambda,u}) - \partial_{a}b(u,\theta_t^{u})\big]  \delta_t^u \d \lambda\;,\\
    I_t^{\epsilon,u,2} = &\int_{0}^{1}  \big[\partial_{x}b(u,\theta_t^{\epsilon,\lambda,u})-\partial_{x} b(u,\theta_t^{u})\big] V_t^u  \d \lambda\;,\\
    I_t^{\epsilon,u,3} = & \int_{0}^{1} \int_{I} \mathbb{\tilde{E}}\Big[\big(\partial  \frac{\delta }{\delta m}b(u,\P_{X_t^{\epsilon,\lambda,.}},{X}_t^{\epsilon,\lambda,u},\alpha_t^{\epsilon,\lambda,u})(\tilde{u},\tilde{X}_t^{\epsilon,\lambda,\tilde{u}})-\partial  \frac{\delta }{\delta m}b(u,\P_{X_t^{.}},X_t^{u},\alpha_t^{u})(\tilde{u},\tilde{X}_t^{\tilde{u}})\big) \tilde{V}_t^{\tilde{u}} \Big] \d \tilde{u}  \d \lambda\;.
\end{align}
From \eqref{dynVeps} and since $V_0^{\eps,u}=0$ we have using Young inequality
%
\begin{align}
    |V_t^{\epsilon,u}|^2 &\leq C \bigg( \big|\int_{0}^{t} V_s^{\epsilon,u,1} \d s \big|^2 + \big|\int_{0}^{t} V_s^{\epsilon,u,2} \d W_s^u  \big|^2 \bigg), 
\end{align}
where $C$ is a generic constant, $i.e.$ depending on $T$ but not depending on $\epsilon$, which may vary from line to line. 
Taking the supremum and  the expectation and now using Jensen and BDG's inequality, we therefore have 
\begin{align}
     \mathbb{E} \big[\underset{0 \leq t \leq r}{\text{ sup}} |V_t^{\epsilon,u}|^2 \big] ~\leq~ C \bigg( \int_{0}^{r}  \mathbb{E} \big[ |V_t^{\epsilon,u,1}|^2 \big] \d t + | \int_{0}^{r} \mathbb{E} \big[ |V_t^{\epsilon,u,2}|^2 \big] \d t \bigg),
\end{align}
for $r \in [0,T]$. Using Fubini-Tonelli theorem we get
\begin{align}
     \int_{I} \mathbb{E} \big[\underset{0 \leq t \leq r}{\text{ sup}} |V_t^{\epsilon,u}|^2 \big] \d u \leq  C \bigg( \int_{0}^{r} \int_{I} \mathbb{E} \big[ |V_t^{\epsilon,u,1}|^2 \big] \d u \d t + \int_{0}^{r} \int_{I} \mathbb{E} \big[ | V_t^{\epsilon,u,2}|^2 \big] \d u \d t \bigg),
\end{align}
for $r \in [0,T]$. From Assumption \ref{assumptionbasic} and Cauchy Schwarz inequality we have
\begin{align}
   \int_0^T \int_I \E\big[|I_t^{\epsilon,u,i}|^2 \big] \d u \d t~\leq~C_\eps    \int_0^T\int_I \E\big[|V_t^{u}|^2+|\alpha_t^{u}|^2+|\beta_t^{u}|^2 \big] \d u \d t~:=~\bar C_\eps,
\end{align}
for $i \in \lbrace 1,2,3 \rbrace$ where $\lim_{\eps\rightarrow0}C_\eps=0$. This gives from the definition of $V^{\epsilon,u,1}$ and Assumption \ref{assumptionbasic}
\begin{align}
    \int_I \mathbb{E} \big[ \big|V_r^{\epsilon,u,1} \big|^2 \big]  \d u ~\leq~ \bar C_\epsilon  + C\int_{I}\mathbb{E} \big[\underset{0 \leq t \leq r}{\text{sup}} \big|V_t^{\epsilon,u} \big|^2  \big] \d u 
    .
\end{align}
The  term $V_t^{\epsilon,u,2}$ is handled in a similar way after using Jensen and BDG's inequality such that we have a similar estimate. We finally get  
\begin{align}\label{eq : borne sup}
    \int_{I} \mathbb{E} \big[\underset{0 \leq t \leq r}{\text{ sup}} \big|V_t^{\epsilon,u} \big|^2 \big] \d u  \leq 2\bar C_\epsilon+C\int_{0}^{r} \int_{I} \mathbb{E} \big[\underset{0 \leq s \leq t}{\text{ sup}} \big|V_s^{\epsilon,u} \big|^2 \big] \d u \d t,
\end{align}
for $r\in[0,T]$. Applying Grönwall's lemma  and sending $\epsilon$ to 0 we get \reff{derXV}.
\end{proof}

\begin{Lemma}\label{eq : Lemma Gateaux J}
    The function $\alpha \in \Ac \mapsto J(\alpha)$ is  G\^ateaux differentiable with G\^ateaux derivative in the direction  $\beta\in \Ac$ given by
    \begin{align}
     {\lim_{\epsilon \to 0} } \frac{1}{\epsilon} \big( J(\alpha + \epsilon (\beta-\alpha)) - J(\alpha)\big) 
    &= \\
    \int_{I} \mathbb{E}\left[\int_{0}^{T} \partial_{x} f(u,\theta_t^u) \cdot  V_t^u + \partial_{a} f(u,\theta_t^{u}) \cdot (\beta_t^u-\alpha_t^u) + \int_{I} \mathbb{\tilde{E}} \Big[ \partial\frac{\delta}{\delta m} f(u,\theta_t^u)(\tilde{u}, \tilde{X}_t^{\tilde{u}}) \cdot  \tilde{V}_t^{\tilde{u}}\Big] \d\tilde{u} \d t \right] \d u \notag 
    &\\
    + \int_{I} \mathbb{E}\left[ \partial_x g(u,X_T^u, \P_{X_T^.}) \cdot  V_T^u + \int_{I}\mathbb{\tilde{E}}\Big[\partial \frac{\delta }{\delta m}g (u, X_T^u,\P_{X_T^.})(\tilde{u},\tilde{X}_T^{\tilde{u}}) \cdot \tilde{V}_T^{\tilde{u}}\Big] \d\tilde{u}\right] \d u\;. &
\end{align}
\begin{proof} We use the same notations as in the proof of Lemma \ref{lemDerX=V} and we recall that
$\alpha^{\epsilon}=(\alpha^u + \epsilon \beta^u)_{u \in I}$ $(X^{\epsilon,u})_{u\in I} =(X^{\alpha^{\epsilon},u})_{u\in I}$,  $(\theta_t^{\epsilon,u})_{u\in I} = (X^{\epsilon,u}, \P_{X^{\epsilon,.}},\alpha^{\epsilon,u})_{u\in I}$, $(V^{\epsilon,u})_{u\in I} = (\frac{1}{\epsilon}(X^{\epsilon,u} - X^u)- V^u)_{u\in I}$,   $(X^{\epsilon,\lambda,u})_{u\in I} = (X^u + \lambda \epsilon (V^{\epsilon,u} + V^{u}))_{u\in I}$, $(\alpha^{\epsilon,\lambda,u})_{u\in I} = (\alpha^{u} + \lambda \epsilon \beta^u)_{u\in I}$ and $(\theta_t^{\epsilon,\lambda,u})_{u\in I}=(X^{\epsilon,\lambda,u},\P_{X^{\epsilon,\lambda,.}},\alpha^{\epsilon,\lambda,u})_{u\in I}$. 
From the definition of $J(\alpha)$, we have 
\begin{align}
      \underset{\epsilon \to 0}{\text{lim}} \frac{1}{\epsilon} \big( J(\alpha + \epsilon (\beta - \alpha)) - J(\alpha) \big)  &=\\  \underset{\epsilon \to 0}{\text{lim}} \frac{1}{\epsilon} \int_{I}    \mathbb{E} \Big[\int_{0}^{T} \big[f(u,\theta_t^{\epsilon,u}) - f(u,\theta_t^u)\big] \d t 
     +  g(u,X_T^{\epsilon,u},\P_{X_T^{\epsilon,.}}) -  g(u,X_T^{u},\P_{X_T^.}) \Big]   \d u. &
\end{align}
Using Taylor formula we have
\begin{align}
   \lim_{\epsilon \to 0} \int_{I} \frac{1}{\epsilon} \E \Big[\int_{0}^{T} \big(f(u,\theta_t^{\epsilon,u}) - f(u,\theta_t^u) \big) \d t \Big] \d u & = \\
      \lim_{\epsilon \to 0}  \frac{1}{\epsilon} \int_{I} \mathbb{E} \bigg[ \int_{0}^{T} \int_{0}^{1}  \Big(\partial_x f(u,\theta_t^{\lambda,\epsilon,u}) \cdot (V_t^{\epsilon,u} + V_t^u) + \partial_{a} f(u,\theta_t^{\lambda,\epsilon,u}) \cdot  (\beta_t^u - \alpha_t^u)  \notag & \\
   + \int_{I} \mathbb{\tilde{E}}\Big[\partial \frac{\delta }{\delta m}f(u,\theta_{t}^{\lambda,\epsilon,u})(\tilde{u},\tilde{X}_t^{\lambda,\epsilon,\tilde{u}}) \cdot  (\tilde{V}_t^{\epsilon,\tilde{u}} + \tilde{V}_t^{\tilde{u}})\d \tilde{u}\Big]\Big)\d \lambda \d t\bigg] \d u\;. &
\end{align}
From Assumption \ref{assumptionbasic} and Lemma \ref{lemDerX=V} we can apply the dominated convergence and we get 
%
\begin{align}
   \lim_{\epsilon \to 0} \int_{I} \frac{1}{\epsilon} \E \Big[\int_{0}^{T} \big(f(u,\theta_t^{\epsilon,u}) - f(u,\theta_t^u) \big) \d t \Big] \d u  & =\\ \int_{I} \mathbb{E}\bigg[\int_{0}^{T} \partial_{x} f(u,\theta_t^u) \cdot V_t^u + \partial_{a} f(u,\theta_t^{u}) \cdot (\beta_t^u - \alpha_t^u) + \int_{I} \mathbb{\tilde{E}} \Big[ \partial \frac{\delta }{\delta m}f(u,\theta_t^u)(\tilde{u}, \tilde{X}_t^{\tilde{u}}) \cdot \tilde{V}_t^{\tilde{u}} \d \tilde{u}\Big] \d t \bigg] \d u\;. & 
\end{align}
The same arguments give
\begin{align}
   \lim_{\epsilon \to 0} \int_{I} \frac{1}{\epsilon} \E \Big[\big(g(u,X_T^u, \P_{X_T^.}) - g(u,X_T^{\epsilon,u}, \P_{X_T^{\epsilon,.}}) \big)  \Big] \d u  & =\\  \int_{I} \mathbb{E}\left[ \partial_x g(u,X_T^u, \P_{X_T^.}) \cdot  V_T^u + \int_{I}\mathbb{\tilde{E}}\Big[\partial \frac{\delta }{\delta m}g (u, X_T^u,\P_{X_T^.})(\tilde{u},\tilde{X}_T^{\tilde{u}}) \cdot \tilde{V}_T^{\tilde{u}}\Big] \d\tilde{u}\right] \d u &, 
\end{align}
which ends the proof.

\end{proof}
\end{Lemma}

\begin{Lemma}\label{eq : Lemma int YV}
 Let $(Y,Z) = (Y_t^u,Z_t^u)_{u \in I, 0 \leq t \leq T}$ be an adjoint process according to Definition \ref{def : adjoint processes}. We then have
\begin{align}
    \int_{I} \mathbb{E}[Y_T^u \cdot V_T^u] du  &= \int_{I} \mathbb{E} \bigg[\int_{0}^{T} Y_t^u \cdot  \big(\partial_{a} b(u,\theta_t^u)(\beta_t^u - \alpha_t^u) \big)  + Z_t^u : \big(\partial_{a} \sigma(u,\theta_t^u) (\beta_t^u-\alpha_t^u) \big) - \partial_{x} f(u,\theta_t^u) \cdot V_t^u \notag \\
     &\quad- \int_{I} \mathbb{\tilde{E}}\Big[\partial  \frac{\delta }{\delta m}f(u,X_t^u,\P_{X_t^.},\alpha_t^u)(\tilde{u},\tilde{X}_t^{\tilde{u}}) \cdot \tilde{V}_t^{\tilde{u}}\Big] \d \tilde{u}  \d t \bigg] \d u\;. 
\end{align}
\end{Lemma}

\begin{proof}
Let $\Theta_t^u = (X_t^u,\P_{X_t^.}, Y_t^u,Z_t^u,\alpha_t^u)$ for $u\in I$ and $t\in[0,T]$ and $\tilde{\Theta}_t^u = (\tilde{X}_t^{u},\P_{X_t^.},Y_t^u,Z_t^u,\tilde{\alpha}_t^{u})$ where $(\tilde{X},\tilde{\alpha})$ is an independent copy of $(X,\alpha)$ defined on $(\tilde{\Omega},\tilde{\Fc},\tilde{\P})$.   From Itô's formula we have  
\begin{align}
        Y_T^u\cdot  V_T^u~ = ~& Y_0^u \cdot V_0^u + \int_{0}^{T} Y_t^u \cdot \d V_t^u + \int_{0}^{T} \d Y_t^u \cdot  V_t^u + \int_{0}^{T} \d \langle Y^u,V^u \rangle_t  \notag \\
          = ~& M_T^u + \int_{0}^{T} \bigg[ Y_t^u \cdot  \big(\partial_x b(u,\theta_t^u) V_t^u \big)+ Y_t^u \cdot \big(\partial_{a} b(u,\theta_t^u) (\beta_t^u - \alpha_t^u)\big)+ Y_t^u \cdot \int_{I} \mathbb{\tilde{E}}\Big[\partial  \frac{\delta }{\delta m}b(u,\theta_t^u)(\tilde{u},\tilde{X}_t^{\tilde{u}})\tilde{V}_t^{\tilde{u}} \Big] \d \tilde{u}   \notag \\
        &- \partial_x H(u,\Theta_t^u) \cdot  V_t^u  - \int_{I} \mathbb{\tilde{E}} \Big[\partial  \dmu{ H}(\tilde{u},\tilde{\Theta}_t^{\tilde{u}})(u,X_t^u) \cdot V_{t}^{u} \Big] \d \tilde{u}  \notag \\
        &+ Z_t^u : \big(\partial_x \sigma(u,\theta_t^u) V_t^u \big) + Z_t^u : \big(\partial_{a} \sigma(u,\theta_t^u) (\beta_t^u - \alpha_t^u) \big) + Z_t^u : \int_{I} \mathbb{\tilde{E}}\Big[\partial  \frac{\delta }{\delta m}\sigma(u,\theta_t^u)(\tilde{u},\tilde{X}_t^{\tilde{u}})\tilde{V}_t^{\tilde{u}} \Big] \d\tilde{u} \bigg] \d t , \notag 
\end{align}
where $M^u$ denotes a martingale with respect to the filtration $\F^u$ starting from $M_0^u=0$ such that $\E[M_T^u] = 0$. 
Taking the expectation, we get the result by using  Fubini Theorem as  $\tilde \Theta^u$ is a copy of $\Theta^u$ for all $u\in I$ .
\end{proof}
The following result expresses the G\^ateaux derivative using the Hamiltonian $H$. 
\begin{Corollary}\label{eq : corollary} We have
\begin{equation}\label{eq : Gateaux Derivatives J}
    \lim_{\epsilon \to 0} \frac{1}{\epsilon} \big( J(\alpha + \epsilon (\beta- \alpha)) - J(\alpha) \big) = \int_{I} \mathbb{E} \Big[\int_{0}^{T} \big[\partial_{a} H(u,X_t^u,\mathbb{P}_{X_t^.}, Y_t^u, Z_t^u,\alpha_t^u) \cdot( \beta_t^u - \alpha_t^u)\big] \d t \Big]  \d u.   
\end{equation}
\end{Corollary}

\begin{proof}
Using Fubini's theorem we have 
\begin{align}
&\int_{I} \mathbb{E}\bigg[ \partial_x g(u,X_T^u, \P_{X_T^.}) \cdot  V_T^u + \int_{I}\mathbb{\tilde{E}} \Big[\partial \frac{\delta }{\delta m}g (u, X_T^u,\P_{X_T^.})(\tilde{u},\tilde{X}_T^{\tilde{u}}) \cdot  \tilde{V}_T^{\tilde{u}}\Big] \d \tilde{u}\bigg] \d u \notag \\
=&\int_{I} \mathbb{E}\bigg[ \partial_x g(u,X_T^u,\P_{X_T^.}) \cdot  V_T^u + \int_{I}\mathbb{\tilde{E}} \Big[\partial \frac{\delta }{\delta m}g (\tilde{u}, \tilde{X}_T^{\tilde{u}},\P_{X_T^.})(u,X_T^u) \cdot V_T^{u} \Big] \d \tilde{u}\bigg] \d u \notag \\
=&\int_{I} \mathbb{E}[Y_T^u \cdot V_T^u] \d u \notag.
\end{align}
Then, the result follows from Lemmata \ref{eq : Lemma Gateaux J} and \ref{eq : Lemma int YV}. 
\end{proof}

\begin{Theorem}\label{Theorem : Neccessaryr Condition}
Let Assumption \ref{assumptionbasic} hold and  assume that the Hamiltonian  $H$ is a convex function in its last variable, i.e. the function $a \in A\mapsto H(u,x,\mu,y,z,a)$ is convex for any $(u,x,\mu,y,z)\in I\times \R^d\times L^2(\Pc_2(\R^d))\times \R^d\times\R^{d \times n}$. Let 
$\alpha=(\alpha_t^u)_{u \in I, 0 \leq t \leq T}$
 be an optimal control,  $X=(X_t^u)_{u \in I,0 \leq t \leq T}$ the associated  controlled state process and $(Y,Z)=(Y^u_t,Z^u_t)_{u \in I, 0 \leq t \leq T}$  the associated adjoint processes. 
Then we have 
 \begin{equation}
   H(u,X_t^u,\P_{X_t^.},Y_t^u,Z_t^u,\alpha_t^u)~ \leq~ H(u,X_t^u,\P_{X_t^.},Y_t^u,Z_t^u,a) \quad dt \otimes d\mathbb{P}-a.e. 
 \end{equation}
  for almost every $u \in I$, for all $a\in A$.
\end{Theorem}

\begin{proof}
  Since $\alpha$ is optimal, we have from Corollary \ref{eq : corollary} :
    \begin{equation}\label{ineq-conv}
        \int_{I} \mathbb{E} \Big[\int_{0}^{T} \partial_{a} H(u,X_t^u, \P_{X_t^.}, Y_t^u Z_t^u,\alpha_t^u) \cdot  (\beta_t^u - \alpha_t^u) \d t \Big] \d u \geq 0
    \end{equation}
   for any $\beta\in\Ac$. By the convexity  of the Hamiltonian $H$ with respect to $a \in A$, we have 
    \begin{equation}\label{eq : Diff Hamiltonien}
        \int_{I} \mathbb{E}\left[\int_{0}^{T} \big(H(u,X_t^u, \P_{X_t^.},Y_t^u,Z_t^u,\beta_t^u) - H(u,X_t^u, \P_{X_t^.},Y_t^u,Z_t^u,\alpha_t^u)\big) \d t \right] \d u \geq 0,
    \end{equation} for all $\beta \in \Ac$.
For almost every $u \in I$  and for any $a \in A$, we define the   $\F^u$-progressively measurable set  $C^u \subset [0,T] \times \Omega$ (i.e. such that $C^u \cap [0,t] \in \Bc([0,t]) \otimes \Fc_t^u$   for any $t  \in [0,T]$) as follows 
\begin{align}
   C^u~=~& \Bigg\lbrace (t,\omega) \in [0,T] \times \Omega : \\
    & \qquad H(u,X_t^u(\omega),\P_{X_t^.},Y_t^u(\omega),Z_t^u(\omega),a)  <  H(u,X_t^u(\omega),\P_{X_t^.},Y_t^u(\omega),Z_t^u(\omega),\alpha_t^u(\omega)) \Bigg\rbrace\;.
\end{align}
Then, defining the admissible control process $\beta$ by 
\begin{align}
    \beta^u_t(\omega) = \alpha \mathds{1}(t,\omega)_{ C^u} + \alpha^u_t(\omega) \mathds{1}(t,\omega)_{ (C^u)^c}\;,\quad u\in I\;, (t,\omega)\in [0,T]\times \Omega\;,
\end{align}
we get from \reff{ineq-conv}
\begin{align}
    \mathbb{E}\Big[\int_{0}^{T} \mathds{1}_{C^u} \big(H(u,X_t^u, \P_{X_t^.}, Y_t^u,Z_t^u,a) - H(u,X_t^u, \P_{X_t^.}, Y_t^u,Z_t^u,\alpha_t^u) \big) \d t \Big]  \geq 0, \notag  \quad \text{du-a.e.}
\end{align}
for almost every $u \in I$. We therefore get 
the expected result.
\end{proof}
We also provide another characterization of an optimal control $\alpha$ when the convexity assumption on $A$ and $H$ does not hold.

\begin{Proposition}
 Let Assumption \ref{assumptionbasic} hold and suppose that $A$ is an open subset (not necessarily convex). Let  $\alpha=(\alpha_t^u)_{u \in I, 0 \leq t \leq T}$
 be an optimal control, $X=(X_t^u)_{u \in I,0 \leq t \leq T}$ the associated state process and $(Y,Z)=(Y^u_t,Z^u_t)_{u \in I, 0 \leq t \leq T}$ the associated adjoint processes solving \eqref{eq : adjoint equations}. We have
 \begin{align}
     \partial_{a} H(u,X_t^u,\P_{X_t^.},Y_t^u,Z_t^u,\alpha_t^u)= 0\quad dt \otimes d\P \text{ a.e.}
 \end{align}
 for almost every $u \in I$.
 \end{Proposition}

\begin{proof}
Let $\epsilon_0 > 0$, $\beta \in \R^k$ with $|\beta| = 1$. We define the   $\F^u$-progressively measurable set  $C^u \subset [0,T] \times \Omega$ (i.e. such that $C^u \cap [0,t] \in \Bc([0,t]) \otimes \Fc_t^u$   for any $t  \in [0,T]$) as follows : 
\begin{align}
   C^u~=~& \Bigg\lbrace (t,\omega) \in [0,T] \times \Omega ~: 
    ~ 
    H(u,X_t^u(\omega),\P_{X_t^.},Y_t^u(\omega),Z_t^u(\omega),\alpha_t^u(\omega))\leq 0 \Bigg\rbrace\;.
\end{align}
and consider for any $C \in \Bc(I)$ such that $\beta_t^u(\omega) = \beta \mathds{1}_{\{(t,\omega) \in C^u\}}\mathds{1}_{ \lbrace \text{dist}(\alpha^u_t(\omega),A^c) > \epsilon_0 \rbrace }$ for every $u \in I$ and $t \in [0,T]$. By definition,  we have  $\alpha_t^u + \epsilon \beta_t^u \in A $ for  $u \in I$, $t \in [0,T]$ and $\epsilon \in [0,\epsilon_0[$. Then, from Corollary \ref{eq : corollary}, we have
\begin{align}
   \int_{I} \mathbb{E} \Big[\int_{0}^{T} \partial_{a} H(u,X_t^u,\P_{X_t^.},Y_t^u,Z_t^u,\alpha_t^u) \cdot \beta_t^u \d t \Big] \d u \geq 0
\end{align}
From the definition of $\beta_t^u$, we deduce similarly to the proof of Theorem \ref{Theorem : Neccessaryr Condition} that 
\begin{align}
    \mathds{1}_{\text{dist}(\alpha_t^u,A^c) > \epsilon_0} \partial_{a} H(u,X_t^u,\mathbb{P}_{X_t^.},Y_t^u,Z_t^u,\alpha_t^u) \cdot \beta  \geq 0 \quad dt \otimes d\P \quad \text{a.e}.
\end{align}
Taking the limit $\epsilon_0\rightarrow 0+$ (along a countable sequence) we get
\begin{align}
    \mathds{1}_{\text{dist}(\alpha_t^u,A^c) > 0}  \partial_{a} H(u,X_t^u,\P_{X_t^.},Y_t^u,Z_t^u,\alpha_t^u) = 0 \quad dt \otimes d\P \quad  \text{a.e}.
\end{align}
for almost every $u \in I$. 
Since $A$ is an open subset and $\beta$ can be chosen arbitrarily in the unit ball of $\R^k$, we get the result.
\end{proof}

\subsection{A sufficient condition}

We now give some conditions that ensure the optimality of a control $\alpha$.

\begin{Theorem}\label{Theorem : Sufficient condition}
   Let $\alpha=(\alpha^u)_{u\in I} \in \Ac$, $X=X^{\alpha}$ the corresponding controlled state process and $(Y,Z)$ the corresponding adjoint processes. Assume that 
   \begin{align}
       &(1) \quad \mathbb{R}^{d} \times L^2 \big(\mathcal{P}_2 (\mathbb{R}^{d})\big) \ni (x,\mu) \mapsto g(u,x,\mu)  \text{ is convex,} \notag \\
       &(2) \quad \mathbb{R}^{d} \times L^2 \big(\mathcal{P}_2(\mathbb{R}^d)\big) \times A \ni (x,\mu,a) \mapsto H(u,x,\mu,Y_t,Z_t,a) \text{ is convex $dt \otimes d\mathbb{P}$ a.e.,}
   \end{align}
    for almost every $u \in I$. Suppose that 
   \begin{align}
       H(u,X_t^u,\P_{X_t^.},Y_t^u,Z_t^u,\alpha_t^u) = \inf_{a \in A} H(u,X_t^u,\P_{X_t^.},Y_t^u,Z_t^u,a), \quad \text{ $dt \otimes d\mathbb{P}$ a.e.}
   \end{align}
   for almost every $u \in I$. Then, $\alpha$ is an optimal control: 
   $J(\alpha) = {\inf}_{\alpha' \in \mathcal{A}} J(\alpha')$.
\end{Theorem}

\begin{proof}
 Let $\hat{\alpha}=(\hat{\alpha}^u)_{u \in I} \in \mathcal{A}$. As previously, we introduce the notation $\hat{\theta}_t^{u}= (\hat{X}_t^{u},\P_{\hat{X}_t^.},\hat{\alpha}_t^u)$. We also write  
 $\hat{\Theta}_t^u= (\hat{X}_t^u,,\P_{\hat{X}_t^.},Y_t^u,Z_t^u,\hat{\alpha}_t^u)$. We then have 
    \begin{align}
        J(\alpha) - J(\hat{\alpha}) &= \int_{I} \mathbb{E}\Big[g(u,X^u_T,\P_{X_T^.}) - g(u,\hat{X}_T^{u},\P_{\hat{X}_T^.}) \Big] \d u + \int_{I} \mathbb{E} \Big[\int_{0}^{T} \big(f(u,\theta_t^u) - f(u,\hat{\theta}_t^u)\big) \d t \Big] \d u \notag \\
        &= \int_{I} \mathbb{E}\Big[g(u,X^u_T,\P_{X_T^.}) - g(u,\hat{X}_T^{u},\P_{\hat{X}_T^.}) \Big] \d u +  \int_{I} \mathbb{E} \Big[\int_{0}^{T} \big(H(u,\Theta_t^u) - H(u,\hat{\Theta}_t^u)\big) \d t \Big] \d u \notag \\
        &\hspace{0.4 cm}- \int_{I} \mathbb{E}\Big[\int_{0}^{T}\big( b(u,\theta_t^u) - b(u,\hat{\theta}_t^u) \big) \cdot Y_t^u + \big(\sigma(u,\theta_t^u) - \sigma(u,\hat{\theta}_t^u)\big) :Z_t^u \d t \Big] \d u. \notag 
    \end{align}
Recalling the convex hypothesis (1) on $g$ in the sense \eqref{eq : convex case 2}, 
we have  
\begin{align}\label{eq : convexity inequality on g}
    g(u,x,\mu) - g(u,\hat{x},\hat{\mu}) \leq (x-\hat{x}) \cdot  \partial_x g(u,x,\mu) + \int_{I} \mathbb{\tilde{E}} [\partial \frac{\delta }{\delta m}g(u,x,\mu)(\tilde{u},\tilde{X}^{\tilde{u}}) \cdot (\tilde{X}^{\tilde{u}} -\tilde{\hat{X}}^{\tilde{u}})] \d \tilde{u}\;.
\end{align}
Therefore, by taking expectation and using  Fubini's theorem, we get 
\begin{align}\label{eq : Convexity Inequality}
    \mathbb{E}\big[g(u,X_T^u,\P_{X_T^.}) - g(u,\hat{X}_T^{u},\P_{\hat{X}_T^.})\big] ~\leq ~& \mathbb{E}\big[\partial_x g(u,X_T^u,\P_{X_T^.}) \cdot  (X_T^u - \hat{X}_T^u) \big] \notag \\
    &+ \mathbb{E} \bigg[\int_{I} \mathbb{\tilde{E}} \Big[\partial \frac{\delta }{\delta m}g(u,X_T^u,\P_{X_T^.})(\tilde{u},\tilde{X}_T^{\tilde{u}}) \cdot (\tilde{X}_T^{\tilde{u}} -\tilde{\hat{X}}_T^{\tilde{u}}) \Big] \d \tilde{u}\bigg] \notag \\
    ~\leq~ &\mathbb{E} [Y_T^u \cdot (X_T^u - \hat{X}_T^u)] \;.
\end{align}
By Itô's formula, we have  
\begin{align}
    \mathbb{E}[(X_T^u - \hat{X}_T^u) \cdot Y_T^u]  
    ~=~& \mathbb{E} \Big[\int_{0}^{T} (X_t^u-\hat{X}_t^u) \cdot \d Y_t^u + \int_{0}^{T} Y_t^u \cdot \d [X_t^u - \hat{X}_t^u] + \int_{0}^{T} \big(\sigma(u,\theta_t^u) - \sigma(u,\hat{\theta}_t^u)\big) :Z_t^u \d t \Big] \notag \\
    ~=~&\mathbb{E}[M_T^u]- \mathbb{E} \bigg[\int_{0}^{T} \Bigg((X_t^u - \hat{X}_t^u) \cdot  \partial_x H(u,\Theta_t^u) - \int_{I} \mathbb{\tilde{E}} \Big[\partial \frac{\delta }{\delta m}H(\tilde{u},\tilde{\Theta}_t^{\tilde{u}})(u,X_{t}^{u}) \cdot (X_t^{u} - \hat{X}_t^{u}) \Big] \d \tilde{u}  \notag \\
    &\qquad \qquad +   \big(b(u,\theta_t^u) - b(u,\hat{\theta}_t^u)\big) \cdot  Y_t^u + \big(\sigma(u,\theta_t^u) - \sigma(u,\hat{\theta}_t^u)\big) :Z_t^u  \Bigg)\d t\bigg],
\end{align}
where $M^u$ denotes the square integrable martingale defined by 
\begin{align}
M_s^u = \int_{0}^{s} Y_t^u \cdot  \left[\sigma(u,\theta_t^u) - \sigma(u,\hat{\theta}_t^u) \right] \d W_t^u + \int_{0}^{s} (X_t^u - \hat{X}_t^u) \cdot Z_t^u \d W_t^u\;,\quad s\in[0,T].    
\end{align}
Therefore, using the convexity assumption (2) on $H$ we get 
\begin{align}
    J(\alpha) - J(\hat{\alpha}) &\leq   \int_{I} \mathbb{E} \Big[\int_{0}^{T} \big(H(u,\Theta_t^u) - H(u,\hat{\Theta}_t^u)\big) \d t \Big] \d u \notag \\
    &- \int_{I} \mathbb{E} \bigg[\int_{0}^{T} \Big(\partial_x H(u,\Theta_t^u) \cdot  (X_t^u -\hat{X}_t^u) +\int_{I}  \mathbb{\tilde{E}} \Big[\partial \frac{\delta }{\delta m}H(\tilde{u},\tilde{\Theta}_t^{\tilde{u}})(u,X_{t}^{u}) \cdot (X_t^{u} - \hat{X}_t^{u}) \d \tilde{u} \Big]\Big) \d t \bigg] \d u \notag \\
    &\leq \int_{I} \mathbb{E} \Big[\int_{0}^{T} \partial_{a} H(u,\Theta_t^u) \cdot  (\alpha_t^u - \hat{\alpha}_t^u) \d t \Big] \d u\;.
\end{align}
From the convexity assumption on $H$ and since $\alpha$ is a minimum point of $H$, we have 
\begin{align}
\partial_{a} H(u,\Theta_t^u) \cdot (\alpha_t^u - a) \leq  0 \quad  dt \otimes d\mathbb{P}-a.e.    
\end{align}
 for almost every $u \in I$ and every $a \in A$,
which ends the proof. 
\end{proof}

\section{Solvability of the Pontryagin collection of FBSDE}

In this section, we will discuss the application of the stochastic maximum principle to the solution of the mean field control of non exchangeable SDE. As usual, the strategy is to identify  a minimizer of the Hamiltonian and use it in the forward dynamics and in the adjoint equations. Therefore, we get a collection of FBSDE due to the non exchangeable mean-field system. We establish existence and uniqueness of the system   by applying a continuation method for the collection of FBSDE as in \cite{carmona_probabilistic_2018-1}.

\subsection{Assumptions}

In this section, we now introduce assumptions that ensure the existence and uniqueness of the collection of FBSDE. 
\begin{Assumption}\label{Assumption : Existence and Uniqueness} There exists two constants $L \geq 0$ and $\lambda > 0$ such that  
\noindent 
 
\begin{itemize}
\item[(i)]  The drift $b$ and the volatility $\sigma$ are linear in $\mu$, $x$ and $\alpha$ such that 
    \begin{align}
        b(u,x,\mu,a) &= b_0(u) +   \int_{I} b_1(u,v) \bar{\mu}^v \d v + b_2(u) x + b_3(u) a, \notag \\
        \sigma(u,x,\mu,a)&=\sigma_0(u) +\int_{I} \sigma_1(u,v) \bar{\mu}^v \d v + \sigma_2(u)x + \sigma_3(u) a ,
    \end{align}
for some bounded measurable deterministic functions $b_0,b_1,b_2,b_3$ with values in $\R^d , \R^{d \times d},\R^{d \times d}, \R^{d \times m}$ and $\sigma_0,\sigma_1,\sigma_2,\sigma_3$ with values in $\R^{d \times n}, \R^{(d \times n) \times d},\R^{(d \times n) \times d}$ and $\R^{(d \times n) \times m}$ and where we used the notation $\bar{\mu}^v = \int_{\R^d} x \mu^v(\d x)$ for the mean of the measure $\mu^v$.

\item [(ii)] There exists a constant $L$ such that
%
%
%
\begin{align}
    \big| \partial_x f(u,x',\mu',a') - \partial_x f(u,x,\mu,a) \big| &\leq L \bigg(|x-x'| + |a - a'|+ \bd(\mu,\mu') \bigg),  \notag \\
    \big| \partial_x g(u,x',\mu') - \partial_x g(u,x,\mu) \big| &\leq L \bigg(|x-x'| + \bd(\mu,\mu') \bigg),
\end{align}
and
\begin{align}
    \big| \partial_a f(u,x',\mu',a') - \partial_a f(u,x,\mu,a) \big| &\leq L \bigg(|x-x'| + |a - a'|+ \bd(\mu,\mu') \bigg),  \notag \\
    \big| \partial_a g(u,x',\mu') - \partial_a g(u,x,\mu) \big| &\leq L \bigg(|x-x'| + \bd(\mu,\mu') \bigg),
\end{align}
for all $u \in I$,  $x,x' \in \R^d$, any $a,a' \in A$ and any $\mu=(\mu^u)_{u}, \mu'=(\mu'^{u})_{u} \in L^2(\Pc_2(\R^d))$.

\item [(iii)] There exists a constant $L$ such that 
\begin{align}
     & \int_{I} \mathbb{E} \Big[ \big| \partial \frac{\delta }{\delta m}f(u,x',\mu',a')(\tilde{u},X'^{\tilde{u}}) - \partial \frac{\delta }{\delta m}f(u,x,\mu,a)(\tilde{u},X^{\tilde{u}}) \big|^2 \Big] \d \tilde{u}  \\
    \leq ~~& L \left(|x'-x|^2  + |a' - a|^2 + \int_{I} \mathbb{E} \Big[\big|X'^u - X^u \big|^2 \Big] \d u\right),
\end{align}
and
\begin{align}
      & \int_{I} \mathbb{E} \Big[ \big| \partial \frac{\delta}{\delta m} g(u,x',\mu')(\tilde{u},X'^{\tilde{u}}) - \partial \frac{\delta }{\delta m}g(u,x,\mu)(\tilde{u},X^{\tilde{u}}) \big|^2 \Big] \d \tilde{u} \\
    \leq~~ & L \left( |x'-x|^2  + \int_{I} \mathbb{E} \Big[ \big|X'^u - X^u \big|^2 \Big] \d u\right),
\end{align}
for all $u \in I$, any $x,x' \in \R^d$ ,  $a,a' \in A$,  $\mu=(\mu^u)_{u}, \mu'=(\mu'^{u})_{u} \in L^2(\Pc_2(\R^d))$, and $X^u$ and $X'^{u}$ $\R^d$-valued random variables such that $X^u\sim \mu^u$ and $X'^{u} \sim \mu'^{u}$ and $u\mapsto \P_{(X^u,X'^u)}$ is measurable.

\item [(iv)]  The functions $f$ and $g$ satisfy the following convexity property 
\begin{align}
        f(u,x',\mu',a') - f(u,x,\mu,a) - \partial_x f(u,x,\mu,a) \cdot (x'-x) -
         \partial_{a} f(u,x,\mu,a) \cdot (a'-a) \notag & \\ 
         - \int_{I} \mathbb{E} \Big[\partial  \frac{\delta }{\delta m}f(u,x,\mu,a)(\tilde{u},X^{\tilde{u}}) \cdot (X'^{\tilde{u}}-X^{\tilde{u}}) \Big] \d \tilde{u} 
         & \geq \lambda |a' - a|^2,
    \end{align}
    and
       \begin{align}
       g(u,x',\mu') - g(u,x,\mu) - \partial_x g(u,x,\mu) \cdot (x'-x) -  \int_{I} \mathbb{E} \Big[\partial \frac{\delta }{\delta m}g(u,x,\mu)(\tilde{u},X^{\tilde{u}}) \cdot (X'^{\tilde{u}}-X^{\tilde{u}}) \Big] \d \tilde{u} \geq 0
   \end{align}
    for all $u \in I$, $(x,\mu,a) \in \R^d \times L^2(\Pc_2(\R^d)) \times A$ and $(x',\mu',a') \in \R^d \times L^2(\Pc_2(\R^d))  \times A$,  $X^u$ and $X'^{u}$ $\R^d$-valued random variables such that $X^u\sim \mu^u$ and $X'^{u} \sim \mu'^{u}$ and $u\mapsto \P_{(X^u,X'^u)}$ is measurable.

\end{itemize}

\end{Assumption}

\begin{Remark}\label{rmk : Convexity of H}
    Note that the L-convexity property of $f$ holds for the Hamiltonian with this specific form. We indeed have
    \begin{align}
     H(u,x',\mu',y,z,a') &- H(u,x,\mu,y,z,a) - \partial_x H(u,x,\mu,y,z,a) \cdot (x'-x) -
         \partial_{a} H(u,x,\mu,y,z,a) \cdot (a'-a) \notag \\
         &+ \int_{I} \mathbb{E}[\partial \frac{\delta }{\delta m}H(u,x,\mu,y,z,a)(\tilde{u},X^{\tilde{u}}) \cdot (X'^{\tilde{u}}-X^{\tilde{u}})] \d \tilde{u} 
         \geq \lambda |a' - a|^2
\end{align}
for every $u \in I$, $x,x' \in \R^d$, $\mu,\mu' \in L^2\big(\Pc_2(\R^d)\big)$, $y \in \R^d$, $z \in \R^{d \times n}$ and $a,a' \in A$.
\end{Remark}

\subsection{The Hamiltonian and the adjoint equations}

According to the linear form of $b$ and $\sigma$,  the Hamiltonian $H$ takes the following form 

\begin{align}
    H(u,x,\mu,y,z,a) &= \big[b_0(u) + \int_{I} b_1(u,v) \bar{\mu}^v \d v + b_2(u) x + b_3(u) a \big] \cdot  y \notag \\
    &\quad + \big[\sigma_0(u) + \int_{I} \sigma_1(u,v) \bar{\mu}^v \d v + \sigma_2(u)x + \sigma_3(u) a \big]:z \notag \\
    &\quad + f(u,x,\mu,a)\;.
\end{align}
Given $(u,x,\mu,y,z) \in I \times \R^d \times L^2(\Pc_2(\R^d)) \times \R^d \times \R^{d \times n}$, the function $A \ni a \mapsto H(u,x,\mu,y,z,a)$ is strictly convex because of the $L$-convexity of $H$ provided by Remark \ref{rmk : Convexity of H}. Hence,  there exists a unique minimizer $\hat{\mra}(u,x,\mu,y,z)$ 
\begin{align}\label{eq : optcontrol minimzor}
    \hat{\mra}(u,x,\mu,y,z) := \text{arg}\min_{a \in A} H(u,x,\mu,y,z,a).
\end{align}
The following lemma states some properties on the optimizer $\hat{\mra}$.

\begin{Lemma}\label{Lemma : Optimizer information}
   The function $I \times \R^d \times L^2\big(\Pc_2(\R^d)\big) \times \R^d \times \R^{d \times n} \ni (u,x,\mu,y,z) \mapsto \hat{\mra}(u,x,\mu,y,z) \in A$ is Borel mesurable and locally bounded, Moreover, for an arbitrary point  $\beta_0$ in $A$,  we have
   \begin{align}
       | \hat{\mra}(u,x,\mu,y,z)| \leq \frac{1}{\lambda} \big( |\partial_{a} f(u,x,\mu,\beta_0) | + |b_3(u)| |y| +  |\sigma_3(u)| |z| \big) + | \beta_0|
   \end{align}
  for any  $(u,x,\mu,y,z) \in I \times \R^d \times L^2 \big(\Pc_2(\R^d)\big) \times \R^d \times \R^{d \times n}$.
\end{Lemma}

\begin{proof}
The Borel measurability of $\hat{\mra}$ follows from the measurability of $H$ and from the strict convexity of $H$ as $\hat{\mra}$ can be approximated through a gradient descent algorithm. For the upper bound, we  can write by noting $\beta_0 \in A$ an arbitrary point and using Remark \ref{rmk : Convexity of H} 
    \begin{align}
        H(u,x,\mu,y,z,\beta_0) \geq & H(u,x,\mu,y,z,\hat{\mra}(u,x,\mu,y,z)) \notag \\
        \geq & H(u,x,\mu,y,z,\beta_0) + \big(\hat{\mra}(u,x,\mu,y,z) - \beta_0 \big) \cdot \partial_{a} H(u,x,\mu,y,z,\beta_0)\\
         & + \lambda | \hat{\mra}(u,x,\mu,y,z) - \beta_0 |^2
    \end{align}
    Since $\partial_{a} H(u,x,\mu,y,z,\beta_0)  = \partial_{a} f(u,x,\mu,y,z,\beta_0) + b_3(u)y + \sigma_3(u)z$, we get from Cauchy-Schwarz inequality 
    \begin{align}
        | \hat{\mra}(u,x,\mu,y,z) - \beta_0 | \leq \frac{1}{\lambda} \big( | \partial_{a} f(u,x,\mu,\beta_0)|+ |b_3(u)| |y| + |\sigma_3(u)||z|\big) ,
    \end{align}
and then 
\begin{align}
    | \hat{\mra}(u,x,\mu,y,z)| \leq \frac{1}{\lambda} \big( | \partial_{a} f(u,x,\mu,\beta_0)| + |b_3(u)| |y| + |\sigma_3(u)||z|\big)  + |\beta_0|.
\end{align}
The local boundedness now follows from the local  boundedness of $\partial_{a} f$ and from the boudedness of $b_3$ and $\sigma_3$.
\end{proof}

Given the necessary and sufficient conditions proven in the previous section, we aim at using the control $\hat{\alpha}=(\hat{\alpha}^u)_{u \in I}$ defined by $\hat{\alpha}^u_t: = \hat{\mra}(u,X_t^u,\mathbb{P}_{X_t^.},Y_t^u,Z_t^u)$ where $\hat{\mra}$ is the minimizer function of the Hamiltonian in \eqref{eq : optcontrol minimzor} and $(X,Y,Z) = (X_t^u,Y_t^u,Z_t^u)_{u \in I , 0 \leq t \leq T}$ is a solution of the following collection of FBSDE

\begin{equation}\label{eq : FBSDE equations LQC}
        \left\{\begin{aligned} \d X_t^u &= \left[b_0(u) +   \int_{I} b_1(u,v) \mathbb{E}\big[X_t^v \big] \d v + b_2(u) X_t^u + b_3(u) \hat{\alpha}_t^u \right] \d t \\
        &\quad + \left[\sigma_0(u) + \int_{I} \sigma_1(u,v) \mathbb{E} \big[X_t^v \big] \d v + \sigma_2(u)X_t^u + \sigma_3(u) \hat{\alpha}_t^u \right] \d W_t^u ,\\
        \d Y_t^u &= - \Big[b_2(u)^{\top} Y_t^u + \sigma_2(u)^{\top}Z_t^u + \partial_x f(u,X_t^u,\mathbb{P}_{X_t^.},\hat{\alpha}_t^u) \Big] \d t   \\
        &\quad - \bigg[ \int_{I}  b_1(v,u)^{\top} \mathbb{E} \big[Y_t^v \big] \d v + \int_{I}  \sigma_1(v,u)^{\top} \mathbb{E} \big[Z_t^v \big] \d v + \int_{I} \mathbb{\tilde{E}} \Big[\partial \frac{\delta }{\delta m}f(\tilde{u},\tilde{X}_t^{\tilde{u}},\mathbb{P}_{X_t^.},\hat{\alpha}_t^u)(u,X_t^u) \Big] \d \tilde{u} \bigg] \d t  \\
        &\quad + Z_t^u \d W_t^u, \\
        X_0^u &= \xi^u,  \\
        Y_T^u &=  \partial_x g(u,X_T^u, \P_{X_T^.}) + \int_{I}\mathbb{\tilde{E}} \Big[\partial \frac{\delta }{\delta m}g (\tilde{u}, \tilde{X}_T^{\tilde{u}}, \P_{X_T^.})(u,X_T^u) \Big] \d \tilde{u}.
        \end{aligned}
        \right.
\end{equation}
We end this subsection by giving the precise definition of  a solution to the collection of FBSDE \eqref{eq : FBSDE equations LQC}.

\begin{Definition}\label{def : solution FBSDEs}
We will say that $(X,Y,Z) = (X_t^u,Y_t^u,Z_t^u)_{u \in I, 0 \leq t \leq T}$ valued in $\R^d \times \R^d \times \R^{d\times n}$ belongs to $\Sc$ if :
    \begin{enumerate}
        \item There exist measurable functions $\mrx$, $\mry$ and $\mrz$ defined on $I \times [0,T] \times \Cc^n_{[0,T]}  \times [0,1]$ such that
        \begin{align}
            X_t^u = \mrx(u,t,W^u_{. \wedge t},Z^u), \quad Y_t^u = \mry(u,t,W^u_{. \wedge t},U^u), \quad Z_t^u = \mrz(u,t,W^u_{. \wedge t},U^u).
        \end{align}
        \item Each process $X^u$ and $Y^u$ are  $\F^u$-adapted and continuous and $Z^u$ is $\F^u$-adapted and square integrable.
        \item The following norm is finite 
        \begin{align}
            \lVert (X,Y,Z) \rVert_{\Sc} : = \left(\int_{I} \mathbb{E}\left[\underset{t \in [0,T]}\sup \big|X_t^u \big|^2 + \underset{t \in [0,T]} \sup \big|Y_t^u \big|^2 + \int_{0}^{T} \big|Z_t^u \big|^2 \d t\right] \d u\right)^{\frac{1}{2}}\;.
        \end{align}
    \end{enumerate}
We say that $(X^u,Y^u,Z^u)_u \in \Sc$ is a solution to \eqref{eq : FBSDE equations LQC} if the equations in \eqref{eq : FBSDE equations LQC} are satisfied for almost every $u \in I$. Moreover, we say that the solution is unique if, whenever $(X^u,Y^u,Z^u)_u$, $(\tilde{X}^u, \tilde{Y}^u, \tilde{Z}^u)_{u} $, the processes $(X^u,Y^u,Z^u)$ and $(\tilde{X}^u,\tilde{Y}^u,\tilde{Z}^u)$ coïncide, up to a $\P$-null set, for almost every $u \in I$.
\end{Definition}

\begin{Remark}\label{rmk : Measurability of X,Y,Z}
    We note that  for $(X^u,Y^u,Z^u)_{u \in I},\; (\bar{X}^u,\bar{Y}^u,\bar{Z}^u)_{u \in I}\in\Sc$, the mapping 
    \begin{align}
        u \in I \mapsto \P_{\big((X^u,Y^u,Z^u),(\bar{X}^u,\bar{Y}^u,\bar{Z}^u)\big)},
    \end{align} is also measurable which makes the space $\Sc$ endowed with $\lVert. \rVert_{\Sc}$ a Banach space and allow for the construction of solutions via Picard iterations on this space. 
\end{Remark}

\subsection{Existence and uniqueness result}

\begin{Theorem}\label{theorem : existence unicity FBSDEs}
Suppose that  Assumptions \ref{assumptionbasic} and \ref{Assumption : Existence and Uniqueness} hold.  Let $\xi = (\xi^u)_{u \in I}$ be an admissible initial condition. Then there exists a unique solution to the collection of forward backward system \eqref{eq : FBSDE equations LQC} $(X,Y,Z) = (X^u,Y^u,Z^u)_u \in \Sc$ in the sense of Definition \ref{def : solution FBSDEs}.
\end{Theorem}

\noindent The proof will follow form similar arguments to those used in the classic mean-field control case, which can be found in Chapter 6 of \cite{carmona_probabilistic_2018}. For the sake of completeness, we  provide the detailed proof. We first need to introduce some additional notations.

As previously done, we use the notation $(\Theta_t^u)_{u \in I, 0 \leq t \leq T}$ for the collection of processes $(X_t^u,\P_{X_t^.},Y_t^u,Z_t^u,\alpha_t^u)_{u \in I,0 \leq t \leq T}$ valued in $\R^d \times L^2 (\Pc_2(\R^d) ) \times \R^d \times \R^{d \times n} \times A$.  We denote by $\mathfrak{S}$ the space of collection of processes   $\Theta = (\Theta_t^u)_{u \in I , 0 \leq t \leq T}$ such that $(X,Y,Z) \in \Sc$ and with  
\begin{align}
    \lVert \Theta \rVert_{\mathfrak{S}} := \left(\int_{I} \E \Big[ \underset{ t \in [0,T]}{\text{ sup }} \big|X_t^u \big|^2 + \underset{t \in [0,T]}{\text{ sup }} \big|Y_t^u \big|^2 + \int_{0}^{T} |Z_t^u|^2 + |\alpha_t^u|^2 \d t \Big] \d u\right)^{\frac{1}{2}} < + \infty.
\end{align}
We also denote by $(\theta_t^u)_{u \in I, 0 \leq t \leq T}$ the restriction of the extended processes $(\Theta_t^u)_{u \in I, t \in [0,T]} \in \mathfrak{S}$ to $(X_t^u,\P_{X_t^.},\alpha_t^u)$.

We now define the set of inputs $\I$ as the set of four-tuples $\Ic$  with $\Ic := \left((\Ic_t^{b,u},\Ic_t^{\sigma,u},\Ic_t^{f,u})_{0 \leq t \leq T},\Ic_T^{g,u}\right)_{u \in I}$ with $(\Ic_t^{b,u},\Ic_t^{\sigma,u},\Ic_t^{f,u},\Ic_T^{g,u})$ valued respectively in $\R^d \times \R^{d \times n} \times \R^d \times \R^d$  satisfying the following properties 

\begin{enumerate}
    \item [$\bullet$] there exist measurable functions $\Ik^{b}$, $\Ik^{\sigma}$, $\Ik^{f}$ defined on $ I \times [0,T] \times \Cc^n([0,T]) \times (0,1)$ and $\Ik^{g}$ defined on $I \times \Cc^n([0,T]) \times (0,1)$  such that
     \begin{align}
        \Ic_t^{b,u} = \Ik^{b}(u,t,W^u_{. \wedge t}, U^u), \quad \Ic_t^{\sigma,u} =  \Ik^{\sigma}(u,t,W^u_{. \wedge t},U^u), \quad  \Ic_t^{f,u} = \Ik^{f}(u,t,W^u_{. \wedge t},U^u), \quad \Ic_T^{g,u} = \Ik^{g}(u,W^u_{. \wedge T},U^u),
    \end{align}
    \item [$\bullet$] the following norm is finite  
    \begin{align}
        \lVert \Ic \rVert_{\I} := \left(\int_I \E \Big[| \Ic_T^{g,u}|^2 + \int_{0}^{T} \big[| \Ic_t^{b,u}|^2 + | \Ic_t^{\sigma,u}|^2 + | \Ic_t^{f,u}|^2 \big] \d t  \Big] \d u\right)^{\frac{1}{2}} < + \infty.
    \end{align}
\end{enumerate}

    The aim of an input $\Ic$ is to be put in the dynamics  \eqref{eq : FBSDE equations LQC} with $\Ic^{b,u}$ and $\Ic^{\sigma,u}$ being plugged into the drift and the diffusion coefficients of the forward equation  while $\Ic^{f,u}$ and $\Ic^{g,u}$ into the the driver and the terminal condition of the backward equation.

\begin{Remark}
    Note that the space $\mathfrak{S}$ endowed with the norm $\lVert . \rVert_{\mathfrak{S}}$ is a Banach space which will allow us to construct an adequate contraction for existence and uniqueness in $\mathfrak{S}$ and therefore in $\Sc$.
\end{Remark}

\begin{Definition}\label{def : continuation BSDE}
    For any $\gamma \in [0,1]$ , $\xi = (\xi^u)_{u \in I}$ an admissible initial condition and an input $\Ic \in \mathbb{I}$, the collection of FBSDE 
    
\begin{equation}\label{eq : Continuation FBSDEs}
   \left\{\begin{aligned} \d X_t^u &= \left(\gamma b(u,\theta_t^u) + \Ic_t^{u,b} \right) \d t + \Big( \gamma \sigma(u,\theta_t^u) + \Ic_t^{u,\sigma}\Big) \d W_t^u \\
        \d Y_t^u &= -\bigg(\gamma \Big(\partial_x H(u,\Theta_t^u) + \int_{I} \mathbb{\tilde{E}}\Big[ \partial \frac{\delta }{\delta m}H(\tilde{u},\tilde{\Theta}_t^{\tilde{u}})(u,X_t^u)\Big] \d \tilde{u} \Big) + \Ic_t^{f,u}\bigg) \d t + Z_t^u \d W_t^u \end{aligned}
        \right.
\end{equation} 
with 
\begin{align}\label{eq : control form}
    \alpha_t^u  = \hat{\mra}(u,X_t^u, \P_{X_t^.},Y_t^u,Z_t^u)
\end{align} 
as optimality condition, $X_0^u = \xi^u$ as initial condition and 
\begin{align}
    Y_T^u = \gamma \left(\partial_x g(u,X_T^u,\P_{X_T^.}) + \int_{I} \mathbb{\tilde{E}}\left[ \partial \frac{\delta }{\delta m}g(\tilde{u},\tilde{X}_T^{\tilde{u}})(u,X_T^u)\right] \d \tilde{u} \right) + I_T^{g,u}\
\end{align}
as terminal condition is referred to as $\Ec(\gamma,\xi,\Ic)$.

\noindent When $(X,Y,Z) \in \Sc$  is a solution in the sense of Definition \ref{def : solution FBSDEs}, the full process $\Theta= (\Theta_t^u)_{u \in I, 0 \leq t \leq T}$ is referred to as the extended solution.
\end{Definition}

\begin{Remark}
    When $\gamma = 1$ and $\Ic \equiv 0$ , the equation \eqref{eq : Continuation FBSDEs} coïncides with the equation \eqref{eq : FBSDE equations LQC}.
\end{Remark}

\begin{Definition}\label{def : Test}
    For any $\gamma \in [0,1]$, we will say  that $(S_{\gamma})$ holds if  for any  $\xi = (\xi^u)_{u \in I}$ an admissible initial condition  and $\Ic \in \I$, $\Ec(\gamma,\xi,\Ic)$ admits a unique extended solution  in $\mathfrak{S}$.
\end{Definition}

\begin{Lemma}\label{eq : Control Difference Theta}
    Let $\gamma \in [0,1]$ such that $(S_{\gamma})$ holds. There exists a constant  $C$ that does not depend on $\gamma$ such that for any $\xi=(\xi^u)_u$, $\bar{\xi} = (\bar{\xi}^u)_{u \in I}$ admissible initial conditions and $\Ic,\bar{\Ic} \in \mathbb{I}$, the respective extended solutions $\Theta$ and $\bar{\Theta}$ of $\mathcal{E}(\gamma,\xi,\Ic)$ and $\mathcal{E}(\gamma,\bar{\xi},\bar{\Ic})$ satisfy 
    \begin{align}
        \lVert \Theta - \bar{\Theta} \rVert_{\mathfrak{S}} \leq  C \Big( \big(\int_{I} \mathbb{E} \big[ |\xi^u - \bar{\xi}^u |^2 \big] \d u\big)^{\frac{1}{2}} + \lVert \Ic - \bar{\Ic} \rVert_{\I}  \Big)\;.
    \end{align}
\end{Lemma}

\begin{proof}
   The proof relies on similar arguments to those used in the derivation of the Pontryagin maximum principle. 
   Using Itô's formula, we have 
    \begin{align}
        \mathbb{E}[(\bar{X}_T^u - X_T^u) \cdot Y_T^u] &= \mathbb{E} [(\bar{\xi}^u - \xi^u) \cdot  Y_0^u] \notag \\
        &\hspace{0.4 cm}- \left(\mathbb{E} \Big[\int_{0}^{T} (\bar{X}_t^u - X_t^u) \cdot  \mathcal{I}_t^{f,u} + (\mathcal{I}_t^{b,u} - \bar{\mathcal{I}}_t^{{b},u}) \cdot Y_t^u + (\mathcal{I}_t^{\sigma,u} - \bar{\mathcal{I}}_t^{\sigma,u}) :Z_t^u \d t \Big] \right) \notag \\
        &\hspace{0.4 cm}- \gamma \left( \mathbb{E} \Big[\int_{0}^{T}    \partial_x H(u,t,\Theta_t^u) \cdot (\bar{X}_t^u - X_t^u) + \int_{I} \mathbb{\tilde{E}}\left[ \partial \frac{\delta }{\delta m}H(\tilde{u},t,\tilde{\Theta}_t^{\tilde{u}})(u,X_t^u)\right] \d \tilde{u} \cdot (\bar{X}_t^u - X_t^u) \d t \Big] \right) \notag \\
        &\hspace{0.4 cm}+\gamma \left( \mathbb{E} \Big[\int_{0}^{T}  \left[b(u,t,\bar{\theta}_t^u) - b(u,t,\theta_t^u)\right] \cdot Y_t^u + \left[\sigma(t,u,\bar{\theta}_t^u) - \sigma(t,u,\theta_t^u)\right] :Z_t^u \d t \Big]\right) \notag \\
        &=T_0^u - T_1^u - \gamma T_2^u ,
    \end{align}
where we set 
\begin{align}
    T_0^u &:= \mathbb{E}[(\bar{\xi}^u - \xi^u) \cdot  Y_0^u], \notag \\
    T_1^u &:= \mathbb{E} \Big[\int_{0}^{T} (\bar{X}_t^u - X_t^u) \cdot  \Ic_t^{f,u} + (\Ic_t^{b,u} - \bar{\Ic}_t^{{b},u}) \cdot Y_t^u + (\Ic_t^{\sigma,u} - \bar\Ic_t^{{\sigma},u}) :Z_t^u \d t \Big], \notag \\
    T_2^u &:= \left( \mathbb{E} \Big[\int_{0}^{T}    \partial_x H(u,t,\Theta_t^u) \cdot (\bar{X}_t^u - X_t^u) + \int_{I} \mathbb{\tilde{E}}\left[ \partial \frac{\delta }{\delta m}H(\tilde{u},t,\tilde{\Theta}_t^{\tilde{u}})(u,X_t^u)\right] \d \tilde{u} \cdot (\bar{X}_t^u - X_t^u) \d t \Big] \right)  \notag \\
        &\hspace{0.6 cm}- \left( \mathbb{E} \Big[\int_{0}^{T}  \left[b(u,t,\bar{\theta}_t^u) - b(u,t,\theta_t^u)\right] \cdot Y_t^u + \left[\sigma(t,u,\bar{\theta}_t^u) - \sigma(t,u,\theta_t^u)\right] : Z_t^u] \d t ]\right).
\end{align}
Now, from the definition of $Y_T^u$ in Definition \ref{def : continuation BSDE} and following the convex hypothesis for $g$, we have with similar arguments as done in the proof of Theorem \ref{Theorem : Sufficient condition}  
\begin{align}\label{eq : convex hypothesis g}
    \mathbb{E}[(\bar{X}_T^u - X_T^u) \cdot Y_T^u] &= \gamma \mathbb{E}\left[\left(\partial_x g(u,X_T^u,\P_{X_T^.}) + \int_{I} \mathbb{\tilde{E}}\left[ \partial \frac{\delta }{\delta m}g(\tilde{u},\tilde{X}_T^{\tilde{u}})(u,X_T^u)\right] \d \tilde{u} \right) \cdot (\bar{X}_T^u - X_T^u)\right] \notag \\
    &\hspace{0.4 cm}+ \E[(\bar{X}_T^u - X_T^u) \cdot \Ic_T^{u,g}]\notag \\
    &\leq \gamma \mathbb{E}[g(u,\bar{X}_T^u,\P_{\bar{X}_T^.}) - g(u,X_T^u,\P_{X_T^.})] + \E[(\bar{X}_T^u - X_T^u) \cdot \Ic_T^{u,g}].
\end{align}
Using the convexity property of the Hamiltonian we get by a similar way to the proof of Theorem \ref{Theorem : Sufficient condition}
\begin{align}\label{eq : Inequality J}
    \gamma J(\bar{\alpha}) - \gamma J(\alpha) \geq  \int_{I} \Big[\gamma \lambda \mathbb{E} \big[\int_{0}^{T} | \alpha_t^u - \bar{\alpha}_t^u|^2 \d t\big] + T_0^u -T_2^u + \E \big[(X_T^u - \bar{X}_T^u) \cdot \Ic_T^{u,g} \big] \Big]  \d u\;.
\end{align}
Exchanging  $\alpha$ and $\bar{\alpha}$, denoting by $\bar{T}_0^u$ $\bar{T}_1^u$ and $\bar{T}_2^u$ the corresponding terms in the decomposition of  $\mathbb{E}[( X_T^u - \bar{X}_T^u) \cdot \bar Y_T^u]$ 
and summing both inequalities \eqref{eq : Inequality J}, we have 
\begin{align}
    \int_{I} \Big[2 \gamma \lambda \mathbb{E} \big[\int_{0}^{T} | \alpha_t^u - \bar{\alpha}_t^u|^2 \d t \big] + T_0^u + \bar{T}_0^u -(T_2^u + \bar{T}_2^u) - \mathbb{E}\big[(\mathcal{I}_T^{u,g} - \mathcal{I}_T^{u,\bar{g}}) \cdot (X_T^u - \bar{X}_T^u) \big] \Big] \d u \leq 0.
\end{align}
We notice that we have
\begin{align}
    T_0^u + \bar{T}_0^u = - \mathbb{E}[(\xi^u - \bar{\xi}^u) \cdot (Y_0^u - \bar{Y}_0^u)]\;.
\end{align}
Similarly for the sum $T_2^u + \bar{T}_2^u$, we have 
\begin{align}
    T_2^u + \bar{T}_2^u  =\mathbb{E}\left[\int_{0}^{T} -  (\Ic_t^{f,u}-\bar\Ic_t^{{f},u}) \cdot  \left(X_t^u - \bar{X}_t^u \right) + (\Ic_t^{b,u} - \bar\Ic_t^{{b},u}) \cdot  \left(Y_t^u-\bar{Y}_t^u\right) + (\Ic_t^{\sigma,u} - \bar\Ic_t^{{\sigma},u}) : \left(Z_t^u - \bar{Z}_t^u \right) \d t \right]\;.
\end{align}
Applying Young's inequality, there exists a constant $C$ that does not depend on $\gamma$ such that 
\begin{equation}\label{eq : Estimate 1 Proof Lemma FBSDE}
    \int_{I} \gamma  \mathbb{E} \Big[\int_{0}^{T} \big|\alpha_t^u - \bar{\alpha}_t^u \big|^2 \d t \Big] \d u \leq \epsilon \lVert \Theta - \bar{\Theta} \rVert^2_{\mathfrak{S}} + \frac{C}{\epsilon} \left( \int_{I} \mathbb{E} \big[|\xi^u - \bar{\xi}^u|^2 \big] \d u + \lVert \Ic - \bar{\Ic} \rVert^2_{\I} \right),
\end{equation}
for any $\epsilon > 0$. From standard estimates of FBSDE,  there exists a constant $C$, that does not depend of $\gamma$ such that 
\begin{equation}\label{eq : Estimate 2 Proof Lemma FBSDE}
    \int_{I} \mathbb{E} \Big[\underset{0 \leq t \leq T}{\text{ sup }} \big|Y_t^u-\bar{Y}_t^u \big|^2 + \int_{0}^{T} \big|Z_t^u - Z_t^u \big|^2 \d t \Big] \d u \leq C \gamma \int_{I} \mathbb{E}[\underset{0 \leq t \leq T}{\text{ sup}} \Big[ \big|X_t^u - \bar{X}_t^u \big|^2 + \int_{0}^{T} \big|\alpha_t^u - \bar{\alpha}_t^u \big|^2 \d t \Big] \d u + C \lVert \mathcal{I} - \bar{\mathcal{I}} \rVert^2,
\end{equation}
and 
\begin{equation}\label{eq : Estimate 3 Proof Lemma FBSDE}
    \int_{I} \mathbb{E}[\underset{0 \leq t \leq T}{\text{ sup}} |X_t^u - \bar{X}_t^u|^2] \d u \leq \int_{I} \mathbb{E}[|\xi^u - \bar{\xi}^u|^2] \d u  + C \gamma \int_{I} \mathbb{E}[\int_{0}^{T} |\alpha_t^u - \bar{\alpha}_t^u|^2] \d t \d u + C \lVert \mathcal{I} - \bar{\mathcal{I}} \rVert^2.
\end{equation}
Combining equations \eqref{eq : Estimate 1 Proof Lemma FBSDE}, \eqref{eq : Estimate 2 Proof Lemma FBSDE} and \eqref{eq : Estimate 3 Proof Lemma FBSDE},we deduce as $\gamma \leq 1$
\begin{align}
    \int_{I} \mathbb{E}\left[\underset{0 \leq t \leq T}{\text{ sup}} |X_t^u - \bar{X}_t^u|^2 + \underset{0 \leq t \leq T}{\text{ sup}} |Y_t^u - \bar{Y}_t^u|^2 + \int_{0}^{T} |Z_t^u-\bar{Z}_t^u|^2 \d t\right] \d u \notag \\
    \leq C \gamma \int_{I} \mathbb{E}[\int_{0}^{T} |\alpha_t^u - \bar{\alpha}_t^u|^2 \d t] \d u + C \left(\int_{I} \mathbb{E}[|\xi^u - \bar{\xi}^u|^2 ] \d u + \lVert \mathcal{I} - \mathcal{\bar{I}} \rVert^2 \right)  \notag \\
    \leq C \epsilon \lVert \Theta - \bar{\Theta} \rVert^2_{\mathfrak{S}} + \frac{C}{\epsilon} \Big(\mathbb{E}[ \int_{I} |\xi^u - \bar{\xi}^u|^2] \d u + \lVert \mathcal{I} - \bar{\mathcal{I}} \rVert^2 \Big).
\end{align}
Now,  using again \eqref{eq : Estimate 1 Proof Lemma FBSDE}, we get 

\begin{align}
    \lVert \Theta - \bar{\Theta} \rVert^2_{\mathfrak{S}} \leq C \epsilon  \lVert \Theta - \bar{\Theta} \rVert^2_{\mathfrak{S}} +  \frac{C}{\epsilon} \Big(\mathbb{E}[ \int_{I} |\xi^u - \bar{\xi}^u|^2] \d u + \lVert \mathcal{I} - \bar{\mathcal{I}} \rVert^2 \Big).
\end{align}
Choosing $\epsilon$ small enough gives the  result.
\end{proof}

\begin{Lemma}\label{Lemma : Induction S_gamma}
    There exists $\delta_0 > 0$ such that if  $(S_{\gamma})$ is verified for some $\gamma \in [0,1)$, then $(S_{\gamma + \eta})$ is also verified for any $\eta \in (0,\delta_0]$ satisfying $\gamma + \eta \leq 1$.
\end{Lemma}

\begin{proof}
    Let $\gamma$ such that $(S_{\gamma})$ holds and $\xi=(\xi^u)_u$ an admissible initial condition and $\Ic \in \I$. Then we can define the mapping $\Phi$ from $\mathfrak{S}$ to itself for which fixed points coincide with the solution of $\Ec(\gamma + \eta, \xi, \Ic)$. Let $\Phi$ be defined as follows. Let $\Theta \in \mathfrak{S}$. We denote by $\bar{\Theta}$ the extended solution of the  collection of FBSDE $\Ec(\gamma,\xi,\bar{\Ic})$ with :
    \begin{enumerate}
        \item [$\bullet$] $\bar{\Ic}_t^{u,b} = \eta b(u,\theta_t^u) + \Ic _t^{u,b}$,
        \item [$\bullet$] $\bar{\Ic}_t^{u,\sigma} = \eta \sigma(u,\theta_t^u) + \Ic _t^{u,\sigma}$,
        \item [$\bullet$] $\bar{\Ic}_t^{u,f} = \eta \left(\partial_x H(u,\Theta_t^u) + \int_{I} \mathbb{\tilde{E}}\left[ \partial \frac{\delta }{\delta m}H(\tilde{u},\tilde{\Theta}_t^{\tilde{u}})(u,X_t^u)\right] d\tilde{u} \right) + \Ic _t^{u,f}$,
        \item [$\bullet$] $\bar{\Ic}_T^{u,g} = \eta \left(\partial_x g(u,X_T^u,\P_{X_T^.}) + \int_{I} \mathbb{\tilde{E}}\left[ \partial  \frac{\delta }{\delta m}g(\tilde{u},\tilde{X}_T^{\tilde{u}},\P_{X_T^.})(u,X_T^u)\right] \d \tilde{u} \right) + \Ic _T^{u,g}$.
    \end{enumerate}
    Therefore, we need  to verify that as defined $\bar{\Ic}$ is also in $\I$. As $\Theta \in \mathfrak{S}$, we have the representation of $X_t^u, Y_t^u, Z_t^u$ and $\alpha_t^u$ by measurable functions of the form $.(u,t,W_{. \wedge t}, U^u)$ and similarly for the inputs $\Ic^{u,b},\Ic^{u,\sigma},\Ic^{u,f}$ and $\Ic_T^{u,g}$, which allows us to represent $\bar{\Ic}^{u,b}, \bar{\Ic}^{u,\sigma}, \bar{\Ic}^{u,f}$ and $\bar{\Ic}^{u,g}$ by suitable measurable functions of the form $.(u,t,W_{. \wedge t}, U^u)$. Moreover, from the hypothesis on $b,\sigma$, $f$ and $g$, the integrability assumption is verified and we have $\bar{\Ic} \in \I$. Now,  by assumption  the solution of $\Ec(\gamma,\xi,\bar{\Ic})$ is uniquely defined and belongs to $\mathfrak{S}$ so the mapping $\Phi$ is well defined. Moreover if $\Theta$ is a fixed point of $\Phi$, we clearly have that $\Theta$ is an extended solution of $\Ec(\gamma+ \eta,\xi,\Ic)$. Thus, it suffices to show that $\Phi$ is a contraction for $\eta$  small enough. From Lemma \ref{eq : Control Difference Theta},there exists a constant $C$ that does not depend on $\gamma$ such that 
    \begin{align}
        \lVert \bar{\Theta}^1 - \bar{\Theta}^2 \rVert   \leq C  \lVert \bar{\mathcal{I}}^1 - \bar{\mathcal{I}}^2 \rVert \leq C \eta  \lVert \Theta^1 - \Theta^2 \rVert,
    \end{align}
    which give the result by choosing $\eta$ small enough.
 \end{proof}
Now, we need to  show that $(S_{0})$ holds true which will end the proof of Theorem \ref{theorem : existence unicity FBSDEs} by iterating the Lemma \ref{Lemma : Induction S_gamma} as $\eta$ is independent of $\gamma$.

Let $\xi= (\xi^u)_u$ an admissible initial condition and $\Ic \in \I$.  We have to deal with the following collection of FBSDE. 
\begin{equation}\label{eq : FBSDEs S0}
   \left\{\begin{aligned} \d X_t^u &=  \Ic_t^{u,b}  \d t+ \Ic_t^{u,\sigma}  \d W_t^u,\\
        \d Y_t^u &=- \Ic_t^{f,u} \d t + Z_t^u \d W_t^u, 
        \end{aligned}
        \right.
\end{equation} 
with $X_0^u = \xi^u$ as initial condition and $Y_T^u = \Ic_T^{g,u}$ as terminal condition.
To end the proof, we have to show  the following 

\begin{enumerate}
    \item [$\bullet$] There exists $(\mrx,\mry,\mrz)$ measurable functions defined on $I \times [0,T] \times C^n_{[0,T]} \times (0,1) \to \R^d,\R^d,\R^{d \times n}$ such that $X_t^u =\mrx(u,t,W^u_{. \wedge t},U^u)$, $Y_t^u = \mry(u,t,W^u_{. \wedge t},U^u)$ and $Z_t^u = \mrz(u,t,W^u_{. \wedge t},U^u)$.
    \item [$\bullet$] Each process $X^u$ and $Y^u$ are continous and $Z^u$ is squared integrable.
    \item [$\bullet$] The norm $ \left(\int_{I} \mathbb{E}\bigg[\underset{t \in [0,T]}{\text{sup}} |X_t^u|^2 + \underset{t \in [0,T]}{\text{sup}} |Y_t^u|^2 + \int_{0}^{T} |Z_t^u|^2 \d t + \int_{0}^{T} |\hat{\alpha}_t^u|^2 \d t \bigg] \d u\right)^{\frac{1}{2}}  $ is finite.
\end{enumerate}

\begin{Remark}
    Note that the representation of each process $X^u,Y^u,Z^u$ will directly imply that $\alpha_t^u$ can be represented as a suitable function $\mra$ defined on $I \times [0,T] \times \Cc^n_{[0,T]} \times (0,1)$ because of $\alpha_t^u$ definition in \eqref{eq : control form}
\end{Remark}
As the FBSDEs are decoupled, we can solve it $u$ by $u$ and from classic theory, by defining the following processes 
\begin{align}
    X_t^u &= \xi^u + \int_{0}^{t} \Ic_s^{u,b} \d s + \int_{0}^{t} \Ic_s^{u,\sigma} \d W_s^u, \notag \\
    Y_t^u &= \E[ \Ic_{T}^{g,u} + \int_{t}^{T} \Ic_s^{f,u} \d s | \Fc_t^u],
\end{align}
and $Z_t^u$ being the martingale decomposition of the process $Y_t^u + \int_{0}^{t} \Ic_s^{f,u} \d s$ for $t\in[0,T]$. By classic theory and from standard estimates and using the fact that $\Ic \in \I$ and assuming for now that $u \mapsto \P_{(X^u,Y^u,Z^u)}$ is measurable,  we have the mesurability of $u \mapsto \P_{(X^u,Y^u,Z^u,\hat{\alpha}^u)}$ due to $\hat{\alpha}$ definition, the finiteness of the norm  $\int_{I} \mathbb{E}\Big[\underset{t \in [0,T]}{\text{sup}} |X_t^u|^2 + \underset{t \in [0,T]}{\text{sup}} |Y_t^u|^2 + \int_{0}^{T} |Z_t^u|^2 \d t + \int_{0}^{T} |\alpha_t^u|^2 \d t \Big] \d u < + \infty$ and $X^u,Y^u$ continuous and $Z^u$ squared integrable. For the integrability condition on $\alpha$, it follows directly from Lemma \ref{Lemma : Optimizer information}, the Lipschitz property of $\partial_{a}f$ and the local boundedness of $\partial_{a} f$.

Now, we need to show that there exists measurable functions $\mrx$, $\mry$ and $\mrz$ such that we have a suitable representation of $X_t^u$,$Y_t^u$ and $Z_t^u$ which will also provide the measurability of the mapping $u \mapsto \P_{(X^u,Y^u,Z^u)}$.

\begin{proof}
\mbox{}\\

\textbf{Required property for $X$ :}

By assumption, there exists $\xi, \Ik^{b}$ and $\Ik^{\sigma}$ such that we have 
\begin{align}\label{eq : Representation X}
    X_t^u = \xi(u,U^u) + \int_{0}^{t} \Ik^{b}(u,s,W^u_{. \wedge s},U^u) \d s  + \int_{0}^{t} \Ik^{\sigma}(u,s,W^u_{. \wedge s},U^u) \d W^u_s\;,\quad t\in[0,T]\;.
\end{align}
Now, the result follows as the process $(\int_{0}^{t} \Ik^{\sigma}(u,s,W^u_{. \wedge s},U^u) \d W^u_s)_{0 \leq t \leq T}$ is $\F^u$-adapted s.t the stochastic integral $\int_{0}^{t} \Ik^{\sigma}(u,s,W^u_{. \wedge s},U^u) \d W^u_s$ is hence measurable with respect to $\Fc_t^u$. Moreover, from the limit approximation of the stochastic integral and due to the Borel measurability in $u$ of $\Ik^{\sigma}$, it  can therefore be rewritten as $\tilde \Ik^\sigma(u,t,W^u_{. \wedge t}, U^u)$  for a measurable  map $\tilde \Ik^\sigma$ defined on $I \times [0,T] \times \Cc^n_{[0,T]} \times (0,1) $. Then, the function $\mrx$ defined   as 
\begin{align}
    \mrx(u,t,w,z) := \xi(u,z) + \int_{0}^{t} \Ik^{b}(u,s,w_{. \wedge s},z) \d s + \tilde \Ik^\sigma(u,t,w_{. \wedge t},z),
\end{align}
satisfies the required representation form for $X$. 

\vspace{0.2 cm}

\noindent \textbf{Required property for $Y$ and $Z$ :}

 We now prove the required property for $Y$ and $Z$. To this end, we introduce the collection of processes $\tilde{Y} = (\tilde{Y}_t^u)_{u \in I, 0 \leq t \leq T}$ defined by 

\begin{align}
    \tilde{Y}_t^u :=  \E \Big[  \Ic_T^{g,u} + \int_{0}^{T} \Ic_s^{f,u} \d s |\Fc_t^u \Big],
\end{align}
for $t\in[0,T]$ and $u\in I$. $\tilde Y^u$ is a continuous integrable martingale with respect to $\F^u$ and we can rewrite it as
\begin{align}
    \tilde{Y}_t^u &= \E\Big[ \Ic_T^{g,u}+ \int_{t}^{T} \Ic_s^{f,u} \d s  | \Fc_t^u\Big] + \int_{0}^{t}\Ic_s^{f,u} \d s \notag \\ 
    &= \E[\Ik^g_T(u,W^u_{.\wedge T},U^u) | \Fc_t^u] + \E \Big[ \int_{t}^{T} \Ik^{f}(u,s,W^u_{. \wedge s},U^u) \d s | \Fc_t^u \Big] + \int_{0}^{t} \Ik^{f}(u,s,W^u_{. \wedge s},U^u) \d s, 
\end{align}
Since the process $(W^u_{(s \wedge T) \vee t} - W^u_t)_{0 \leq s \leq T}$ is independant from $\Fc_t^u$, we can write following the proof of Theorem 2.1 in \cite{de2024mean} and defining a measurable function $\tilde{\Ik}^g_T$ defined on $I \times \Cc^n_{[0,t]} \times \Cc^n_{[t,T]} \times (0,1) \times \R^d$ as 

\begin{align}
    \tilde{\Ik}^g_T(u,\tilde{w},\hat{w},z) := \Ik^g_T(u,\tilde{w} \oplus \hat{w} ,z),
\end{align}
we have
\begin{align}
   \E[\Ik^g_T(u,W^u,U^u) | \Fc_t^u]&= \E[\tilde{\Ik}_T^g(u,W^u_{. \wedge t}, (W^u_{(s \wedge T) \vee t} - W^u_t)_{s \in [0,T]},U^u) | \Fc_t^u ] \notag \\
    &= \int_{\Cc^d_{[0,T]}} \tilde{\Ik}_T^g(u,W^u_{. \wedge t}, \omega_{(. \wedge T) \vee t} - \omega_t,U^u) \d\W_T(\omega),
\end{align}
Therefore, the function defined as $\mry(u,t,w,z) := \int_{\Cc^n_{[0,T]}} \tilde{\Ik}_T^g(u,w_{. \wedge t},w',z) \d \W_T(w')$ gives the expected result form for $\E[\Ik^g_T(u,W^u,U^u) | \Fc_t^u]$. Applying the same argument to $ \E \Big[ \int_{t}^{T} \Ik^{f}(u,s,W^u_{. \wedge s},U^u) \d s | \Fc_t^u \Big]$, we end up showing the required form for $Y$ since

\begin{align}\label{eq : Representation Y}
    Y_t^u = \tilde{Y}_t^u - \int_{0}^{t} \Ik^{f}(u,s,W^u_{. \wedge s},U^u) \d s \;,\quad t\in[0,T]\;.
\end{align}

We now prove that the process $(Z_t^u)_{t \geq 0}$ satisfies the required representation property. 
Let $h_n = \frac{1}{n}$. We then have  
\begin{align}
    \frac{1}{h_n} \int_{t- h_n}^t Z_s^u \d s = \frac{1}{h_n} \langle Y_t^u, W_t^u \rangle_n,
\end{align}
where $\langle .,. \rangle_n$ denotes the covariation over the interval $[t - \frac{1}{n},t]$. From Theorem 7.20 in \cite{rogers_diffusions_2000}, we have $ \frac{1}{h_n} \int_{t- h_n}^t Z_s^u \d s \to Z_t^u$ $\P-a.s$ which ends the proof due to the definition of the covariation as the limit of suitable functions $\phi^n$ which have the required form due to the representation of $Y$ we just proved. Therefore, the proof is now complete.
\end{proof}

\section{Application : Linear-Quadratic graphon mean field control}

This section is devoted to an application of the Theorem \ref{theorem : existence unicity FBSDEs} in the LQ case with an application in finance through  systemic risk for heterogeneous banks.

We consider the same LQ-NEMF model as introduced in \cite{de2025linear}, where the state process is given by a collection $X = (X^u)_u$ indexed by $u \in I = [0,1]$, and  controlled through  $\alpha = (\alpha^u)_{u }$ taking values in $A = \mathbb{R}^m$. For simplicity, we assume that the diffusion term is uncontrolled. The dynamics of the system are then given by

\begin{align}
    \d X_t^u &= \Big[\beta^u  + A^u X_t^u  + \int_{I} G_{A}(u,v) \E[X_t^v] \d v  + B^u \alpha_t^u \Big] \d t 
         + \gamma^u  \d W_t^u , \quad t \in [0,T], \notag \\
    X_0^u &= \xi^u, u \in I,
\end{align}
 where $\beta \in L^{\infty}(I ; \R^d)$, $\gamma \in L^{\infty}(I ; \R^{d \times n})$, $A \in L^{\infty}(I,\R^{d \times d}), B \in L^{\infty}(I; \R^{d \times m}), G_A \in L^{\infty}(I \times I ; \R^{d \times d})$, $W^u$ is an $\R^n$-valued standard brownian motion and $\xi = (\xi^u)_u$ is  an admissible initial condition and 
the associated cost functional is given by 

\begin{align}\label{eq : Cost Functional LQ form 1}
    J(\alpha) &:= \int_{I} \E \Big[ \int_{0}^{T} \big[ Q^u \big( X_t^u  - \int_{I} \tilde{G}_{Q}(u,v) \E[X_t^v] \d v \big) \cdot  \big( X_t^u  - \int_{I} \tilde{G}_{Q}(u,v) \E[X_t^v] \d v \big)    \ \notag \\
    &\quad  \quad+  \alpha_t^u \cdot  R^u \alpha_t^u + 2  \alpha_t^u \cdot \Gamma^u X_t^u   + 2 \alpha_t^u \cdot \int_{I} G_I(u,v) \E[X_t^v] \d v \big] \d t \notag \\
    &\quad \quad +  P^u\big(X_T^u - \int_{I} \tilde{G}_{P}(u,v) \E[X_T^v] \d v \big) \cdot  \big(X_T^u - \int_{I} \tilde{G}_{P}(u,v) \E[X_T^v] \d v \big)  \Big] \d u 
\end{align}
where $Q,P \in L^{\infty}(I ; \S^d_+)$, $\Gamma \in L^{\infty}(I; \R^{m \times d})$, $R \in L^{\infty}(I;\S^m_{> +}), \tilde{G}_Q,\tilde{G}_P \in L^{\infty}(I \times I, \R^{d \times d})$ and $G_I \in L^{\infty}(I \times I , \R^{m \times d})$. as done in \cite{de2025linear}, we assume  without  loss of generality that  
\begin{align}
    \tilde{G}_Q(u,v) = \tilde{G}_Q(v,u)^{\top}, \quad \tilde{G}_P(u,v) = \tilde{G}_P(v,u)^{\top}.
\end{align}
for anu $u,v\in I$. In this setting, we easily check that Assumptions \ref{assumptionbasic} and \ref{Assumption : Existence and Uniqueness} are satisfied. This ensures existence and uniqueness for the Pontryagin collection of FBSDE in the space $\Sc$ according to Theorem \ref{theorem : existence unicity FBSDEs}.

From \eqref{eq : Cost Functional LQ form 1} and the definition of $H$ given in \eqref{eq : Hamiltonien}, we have 
\begin{align}
    H(u,x,\mu,y,z,a) &= \Big(\beta^u + A^u x + \int_{I} G_{A}(u,v) \bar{\mu}^v \d v + B^u a \Big) \cdot y + \gamma^u:z \notag \\
    &\quad+  Q^u \big(x - \int_{I} \tilde{G}_{Q}(u,v )  \bar{\mu}^v dv \big) \cdot \big(x - \int_{I} \tilde{G}_{Q}(u,v) \bar{\mu}^v \d v \big) \notag \\
    &\quad + a \cdot R^u a + 2 a \cdot \Gamma^ux + 2a \cdot  \int_{I} G_I(u,v) \bar{\mu}^v \d v, \notag 
\end{align}
for $(u,x,\mu,y,z,a) \in I \times \mathbb{R}^{d} \times L^2 \big(\mathcal{P}^{2}(\mathbb{R}^{d})\big) \times \mathbb{R}^{d} \times \mathbb{R}^{d \times n} \times \mathbb{R}^{m}$. Since $R \in L^{\infty}(I;\S^m_{> +})$, the Hamiltonian  $H$ is  strictly convex in $a$. Therefore, it admits a unique minimizer $\hat{\mra}(u,x,\mu,y,z)$ given by  
\begin{align}
    \hat{\mra}(u,x,\mu,y,z) = - \frac{1}{2}(R^u)^{-1} \left[(B^u)^{\top}y  + 2 \Gamma^u x + 2 \int_{I} G_{I}(u,v) \bar{\mu}^v \d v \right].
\end{align}
Theorefore, FBSDE \eqref{eq : FBSDE equations LQC} reads 
\begin{align}\label{eq : FBSDE LQ remplacé}
     \left\{\begin{aligned}\d X_t^u &=\Big(\beta^u + \big[A^u  - B^u (R^u)^{-1} \Gamma^u \big]X_t^u - \frac{1}{2} B^u (R^u)^{-1} (B^u)^{\top} Y_t^u 
    \notag \\
    &\quad + \int_{I} \Big[G_{A}(u,v)  - B^u(R^u)^{-1} G_I(u,v) \Big]\mathbb{E}[X_t^v] \d v \Big) \d t +\gamma^u  \d W_t^u, \notag \\
    \d Y_t^u &=  - \Big[C_Y^u Y_t^u + 2 C_X^u   X_t^u  + 2\int_{I}  \Phi_{X}(u,v)\E[X_t^v] \d v + \int_{I} \Phi_{Y}(u,v)  \mathbb{E}[Y_t^v] \d v \Big] \d t  \notag\\
    &\quad + Z_t^u \d W_t^u,
    \end{aligned}
    \right.
\end{align}
with initial  and terminal conditions given by 

\begin{align}
    X_0^u = \xi^u,  \text{ and } Y_T^u = 2 \big(P^u X_T^u  +  \int_{I} G_{P}(u,v) \E[X_T^v] \d v \big).
\end{align}
The coefficients $C_X,\;C_Y,\;\Phi_X,\;\Phi_Y\;G_Q$, and $G_P,$ are given by
\begin{equation}
\left\{
\begin{aligned}
    C_Y^u &= (A^u)^{\top} - (\Gamma^u)^{\top} (R^u)^{-1} (B^u)^{\top}, 
    \\
    C_X^u &= Q^u - (\Gamma^u)^{\top} (R^u)^{-1} \Gamma^u, \\
    \Phi_{X}(u,v) &= - (\Gamma^u)^{\top} (R^u)^{-1} G_I(u,v) 
    + \Big[ G_Q(u,v) - G_I(v,u)^{\top} (R^v)^{-1} \Gamma^v \Big] \\
    &\quad - \int_I G_I(w,u)^{\top} (R^w)^{-1} G_I(w,v) \, \mathrm{d}w, 
    \label{eq:Phi_definition} \\
    \Phi_{Y}(u,v) &= G_A(v,u)^{\top}  - G_I(v,u)^{\top} (R^v)^{-1} (B^v)^{\top}, \\
    G_{Q}(u,v) &= \int_{I} Q^w \tilde{G}_Q(w,u)\tilde{G}_Q(w,v) \, \mathrm{d}w 
    - (Q^u + Q^v)\tilde{G}_Q(u,v), \\
    G_P(u,v) &= \int_{I} P^w \tilde{G}_P(w,u)\tilde{G}_P(w,v) \, \mathrm{d}w 
    - (P^u + P^v) \tilde{G}_P(u,v)\;,
\end{aligned}
\right.
\end{equation}
for $u,v\in I$. 

For the LQ-NEMF, they show in \cite{de2025linear} that the value function defined on $[0,T] \times L^2 (\Pc^2(\R^d))$ takes the following form 
    \begin{align}
        v(t,\mu) = \int_I \int_{\R^d} x \cdot K_t^ux \mu^u(\d x) \d u  + \int_I \int_{I} \bar{\mu}^u \cdot \bar{K}_t(u,v) \bar{\mu}^v \d v \d u + 2 \int_I \Lambda_t^u \cdot  \bar{\mu}^u \d u  + \int_I R_t^u \d u,
    \end{align}
with $R \in \Cc^1\big([0,T] ; L^2(I;\R)\big)$. Based on the heuristic connection between the value function of the stochastic optimization problem and $(Y^u,Z^u)_{u \in I}$ as mentionned in Appendix \ref{Appendix : Master Equation FBSDE system}, 
we are now looking for an  ansatz for $Y_t^u$ in the following form 

\begin{equation}\label{eq : Ansatz Y}
    Y_t^u = 2  \big( K_t^u   X_t^u + \int_{I}  \bar{K}_t(u,v) \E[X_t^v] \d v  + \Lambda_t^u \big),
\end{equation}
where $K \in C^1 \big([0,T]; L^{\infty}(I;\S^d_+)\big)$, $\bar{K} \in C^1 \big([0,T],L^2(I \times I; \R^{d \times d})\big)$ and $\Lambda \in C^1 \big([0,T]; L^{2}(I; \R^d)\big)$ are to be determined through Riccati equations.
 Plugging the  ansatz \eqref{eq : Ansatz Y} in the dynamics of $X^u$ lead to a linear SDE that can be explicitly solved. Then, identifying the $X_t^u$ term with the dynamics of $ Y_t^u$ leads to the following Ricatti equation for $K_t^u$  
\begin{align}\label{eq : Standard Ricatti}
 \dot{K}_t^u + \Phi^u(K_t^u) -  M^u(K_t^u)^{\top} (R^u)^{-1} M^u(K_t^u) &= 0, ~t\in[0,T]\;,\notag \\
   K_T^u &= P^u,
\end{align}
where we set
\begin{align}\label{eq : Simplification calculs 1}
    \Phi^u(\kappa) &:= (A^u)^{\top} \kappa  + \kappa A^u + Q^u,\notag \\
    M^u(\kappa)&:=  (B^u)^{\top} \kappa + \Gamma^u  .
\end{align}
for all $\kappa \in \S^d_+$.
\begin{Remark}
    This Riccati is similar to the standard Riccati obtained in \cite{de2025linear} (without the diffusion coefficient that we omit for simplicity) for which we can show existence and uniqueness $u$ by $u$ as done in \cite{de2025linear}.
\end{Remark}

Now, after some tedious but straightforward computations, we get the following non-standard  Ricatti equation for $\bar{K}_t(u,v)$ :
\begin{align}\label{eq : Non Standard Ricatti K}
    \dot{\bar{K_t}}+ \Psi(u,v)\big(K_t,\bar{K_t}) - M^u(K_t^u)^{\top} (R^u)^{-1}V^1(u,v)(K_t,\bar{K}_t) &\notag \\ - V^2(v,u)^{\top}(K_t,\bar{K}_t)(R^v)^{-1}M^v(K_t^v) 
    - \int_I V^2(w,u)^{\top} (R^w)^{-1}V^1(w,v) \d w  & = 0,~t\in[0,T]\;,\\  
    \bar{K}_T(u,v) & = G_H(u,v).
\end{align}
where 
$K_t$ is solution of the Riccati equation \eqref{eq : Standard Ricatti}, 
$M^u$ as defined in \eqref{eq : Simplification calculs 1} and 
\begin{align}\label{eq : Simplification Calculs 2}
    \Psi(u,v)(K_t,\bar{k})&:= K_t^u G_A(u,v) + G_A(v,u)^{\top} K_t^v + (A^u)^{\top} \bar{k}(u,v) + \bar{k}(u,v) A^v + \int_I \bar{k}(u,w) G_A(w,v)  \d w  \notag \\
    &\quad+ \int_I G_A(w,u)^{\top} \bar{k}(w,v) \d w + G_Q(u,v), \notag \\
    V^1(u,v)&:=(B^u)^{\top}\bar{k}(u,v)  + G_I(u,v), \notag \\
    V^2(u,v)&:= (B^u)^{\top} \bar{k}(v,u)^{\top} + G_I(u,v),
\end{align}
 for $\bar{k} \in L^2( I \times I ; \R^{d \times d})$ and $u,v\in I$.

Finally, we get the following non-standard Ricatti equation for $\Lambda_t^u$
\begin{align}\label{eq : Ricatti Lambda}
    \dot{\Lambda}_t^u  + K_t^u \beta^u  + \int_I \bar{K}_t(u,v) \beta^v \d v- M^u(K_t^u)^{\top} (R^u)^{-1}(B^u)^{\top} \Lambda_t^u& \notag \\
    + \int_I G_A(v,u)^{\top} \Lambda_t^v \d v  + \int_I V^2(v,u)(\bar{K}_t)^{\top} (R^v)^{-1}(B^v)^{\top} \Lambda_t^v \d v & = 0,~t\in[0,T], \notag \\
    \Lambda_T^u& = 0,
\end{align}
where we used the definition of $M^u$ and  $V^2$ introduced respectively in  \eqref{eq : Simplification calculs 1} and \eqref{eq : Simplification Calculs 2}.

\begin{Remark}
    We again obtained the similar Riccati equation for $\bar{K}_t(u,v)$ and $\Lambda_t^u$ as in \cite{de2025linear}. Moreover, we can notice that we have a triangular system of Riccati equations in the sense that we need to solve $K_t^u$ to solve $\bar{K}_t(u,v)$ and then solve $\Lambda_t^u$. Moreover, it has been shown in \cite{de2025linear} the existence and uniqueness for the triplet $(K,\bar{K},\Lambda)$ relying to some nonlinear functional analysis.
\end{Remark}

\begin{Remark}\label{eq : Special Case Lambda coeff}
    Note that from the existence and uniqueness of a solution, we have in the special case of $\beta^u = 0$ for all $u \in I$ that the Ricatti equation for $\Lambda$ \eqref{eq : Ricatti Lambda} gives $\Lambda_t^u = 0$ for all  $t \in [0,T]$ and  $u \in I$.
\end{Remark}

\begin{Remark}
    Note that the choice of our ansatz gives the following form for  $Z_t^u =  2 K_t^u \gamma^u$ for all $t \in [0,T]$ and $u \in I$.
\end{Remark}

Now that we have the existence and uniqueness of the triplet $(K,\bar{K},\Lambda)$ and due to the measurability of the coefficients $K^u$, $\bar{K}(u,v)$ and $\Lambda^u$, we just need to check that the following collection of SDE

\begin{align}
\d X_t^u &=\bigg(\beta^u - B^u (R^u)^{-1} (B^u)^{\top} \Lambda_t^u + \Big[A^u - B^u (R^u)^{-1} \big((B^u)^{\top} K_t^u + \Gamma^u \big)\Big] X_t^u \notag \\
        &\quad +\int_{I} \Big[G_{A}(u,v) - B^u(R^u)^{-1} \big((B^u)^{\top} \bar{K}_t(u,v) + G_I(u,v)\big)\Big] \mathbb{E}[X_t^v] \d v \bigg) \d t \notag \\
    &\quad + \gamma^u  \d W_t^u,
\end{align}
is uniquely solvable for almost every $u \in I$ with the process $X^u$ continuous and such that there exists a function $\mrx : I  \times [0,T] \times \Cc^n_{[0,T]} \times (0,1)$ with $X_t^u = \mrx(u,t,W^u_{. \wedge t},U^u)$
with the integrability condition $\int_I \E \Big[ \underset{t \in [0,T]}{\sup} |X_t^u|^2 \Big] du < + \infty$ to ensure that the collection of control $\alpha$ defined by 

\begin{align}\label{eq : Optimal control LQ}
    \alpha_t^u = - (R^u)^{-1} \bigg[ \Big((B^u)^{\top}K_t^u + \Gamma^u \Big) X_t^u +  \int_I \Big((B^u)^{\top} \bar{K}_t(u,v )+G_I(u,v)\Big)\E[X_t^v] \d v + (B^u)^{\top} \Lambda_t^u \bigg]
\end{align}
is optimal.
However, the proof of existence and uniqueness  is the same as the proof of Theorem (2.1) in \cite{de2025linear} with a straightforward extension to the case of continuous time-dependent coefficients as discussed in their Remark (2.1). The main difference compared to our current framework lies in the form of the initial condition $\xi = (\xi^u)_{u \in I}$. However, using a simple induction argument  and relying on the measurability requirements of the model coefficients in the label $u$, we obtain  the desired form for $X_t^u$ as an appropriate Borel measurable function $\mrx$ of $(u, t, W^u_{. \wedge t}, U^u)$ similarly to the proof of Theorem \ref{THMexistUniqX}.

\begin{Remark}
    Note that we recover the exact same form of optimal control \eqref{eq : Optimal control LQ} as in  \cite{de2025linear} in their Theorem (5.1).
\end{Remark}

\appendix

\renewcommand{\thesection}{\Alph{section}}

\section{An heuristic link between the Master equation and the FBSDE system}\label{Appendix : Master Equation FBSDE system}

The ansatz form \eqref{eq : Ansatz Y} in the LQ case can be justified through the heuristic link between the adjoint equations $(Y,Z)=(Y^u,Z^u)_{u \in I}$ and the value function of the stochastic optimization problem related to \eqref{eq : Fonctionnelle Coût}. Indeed, as introduced in \cite{de2025linear} (Corollary 3.1), the value function can be written as 

\begin{align}
    v(t,\mu) = \underset{\alpha \in \mathcal{A}}{\text{ inf }} \Big\lbrace \int_{I} \E \Big[ \int_{t}^{\theta} f(u,X_s^{t,\xi,\alpha,u},\P_{X_s^{t,\xi,\alpha,}},\alpha_s^u) \d s\Big]  + v \big(\theta,\P_{X_{\theta}^{t,\xi,\alpha,.}} \big) \Big\rbrace,
\end{align}
for any $\theta \in [t,T]$ and any  $\xi$ admissible initial condition satisfying $\P_{\xi^u} = \mu^u$ for almost every $u \in I$.
As done in \cite{carmona_probabilistic_2018-1} and under the existence of an optimal control, we can expect to have the following master equation

\begin{align}
    Y_t^u =\Uc(u,t,X_t^u,\P_{X_t^.} \big),
\end{align}
for $(u,t,x,\mu) \in I \times [0,T] \times \R^d \times L^2\big(\Pc_2(\R^d)\big)$ and where the mapping $\Uc$ is defined as 

\begin{align}\label{eq : Heuristic Derivation}
    \Uc(u,t,x,\mu) := \partial \frac{\delta}{\delta m}v(t,\mu)(u,x).
\end{align}
In the LQ-NEMF model, as shown in \cite{de2025linear}, the value function has the form 

\begin{align}\label{eq : LQ value function form}
    v(t,\mu) &= \int_{I} \int_{\R^d} x \cdot K_t^u x \mu^u(dx) \d u +\int_{I} \int_{I} \bar{\mu}^u \cdot K_t(u,v)\bar{\mu}^v  \d v \d u  \notag \\
    &\hspace{0.4 cm}+  2 \int_{I} \Lambda_t^u \cdot \bar{\mu}^u \d u  + \int_I R_t^u \d u,
\end{align}
where $K$, $\bar{K}$, $\Lambda$ and $R$ are solution to a triangular Riccati system. Taking the gradient of the flat derivative for \eqref{eq : LQ value function form} justifies the  choice of ansatz \eqref{eq : Ansatz Y} for $Y$.

\bibliographystyle{plain}

\begin{thebibliography}{10}

\bibitem{de2025linear} 
A. De Crescenzo, F. de Feo and H. Pham. 
\textit{Linear-quadratic optimal control for non-exchangeable mean-field SDEs and applications to systemic risk}. 
arXiv:2503.03318, 2025.

\bibitem{lasry2007mean} 
J.-M. Lasry and P.-L. Lions. 
\textit{Mean field games}. 
Japanese Journal of Mathematics, 2(1):229–260, 2007.

\bibitem{coppini2024nonlinear} 
F. Coppini, A. De Crescenzo and H. Pham. 
\textit{Nonlinear Graphon mean-field systems}. 
To appear in Stochastic Processes and their Applications, 2025.

\bibitem{de2024mean} 
A. De Crescenzo, M. Fuhrman, I. Kharroubi and H. Pham. 
\textit{Mean-field control of non-exchangeable systems}. 
arXiv:2407.18635, 2024.

\bibitem{Yong2013} 
J. Yong. 
\textit{Linear-Quadratic Optimal Control Problems for Mean-Field Stochastic Differential Equations}. 
SIAM Journal on Control and Optimization, 51(4):2809–2838, 2013.

\bibitem{bayraktar2023graphon} 
E. Bayraktar, S. Chakraborty and R. Wu. 
\textit{Graphon mean field systems}. 
The Annals of Applied Probability, 33(5):3587–3619, 2023.

\bibitem{lovasz_large_2010} 
L. Lovász. 
\textit{Large Networks and Graph Limits}. 
American Mathematical Society, 2010.

\bibitem{jabin2025mean} 
P.-E. Jabin, D. Poyato and J. Soler. 
\textit{Mean-field limit of non-exchangeable systems}. 
Communications on Pure and Applied Mathematics, 78(4):651–741, 2025.

\bibitem{peng_general_1990} 
S. Peng. 
\textit{A general stochastic maximum principle for optimal control problems}. 
SIAM J. Control Optim., 28(4):966–979, 1990.

\bibitem{yong_stochastic_1999} 
J. Yong and X. Y. Zhou. 
\textit{Stochastic Controls}. 
Springer, 1999.



\bibitem{carmona_probabilistic_2018} 
R. Carmona and F. Delarue. 
\textit{Probabilistic Theory of Mean Field Games with Applications I}. 
Springer, 2018.

\bibitem{carmona_probabilistic_2018-1} 
R. Carmona and F. Delarue. 
\textit{Probabilistic Theory of Mean Field Games with Applications II}. 
Springer, 2018.

\bibitem{lacker2023label} 
D. Lacker and A. Soret. 
\textit{A label-state formulation of stochastic graphon games and approximate equilibria on large networks}. 
Mathematics of Operations Research, 48(4):1987–2018, 2023.


\bibitem{rogers_diffusions_2000} 
L. C. G. Rogers and D. Williams. 
\textit{Diffusions, Markov Processes and Martingales: Volume 2: Itô Calculus}. 
Cambridge University Press, 2000.


\bibitem{carmona_mean_2015} 
R. A. Carmona, J. P. Fouque and L. H. Sun. 
\textit{Mean field games and systemic risk}. 
Communications in Mathematical Sciences, 13(4):911–933, 2015.


\bibitem{bayraktar2023propagation} 
E. Bayraktar, R. Wu and X. Zhang. 
\textit{Propagation of chaos of forward–backward stochastic differential equations with graphon interactions}. 
Applied Mathematics \& Optimization, 88(1):25, 2023.

\bibitem{cao_probabilistic_2025} 
Z. Cao and M. Laurière. 
\textit{Probabilistic Analysis of Graphon Mean Field Control}. 
arXiv:2505.19664, 2025.

\bibitem{basei2019weak} 
M. Basei and H. Pham. 
\textit{A weak martingale approach to linear-quadratic McKean–Vlasov stochastic control problems}. 
Journal of Optimization Theory and Applications, 181:347–382, 2019.

\bibitem{carmona_forwardbackward_2015} 
R. Carmona and F. Delarue. 
\textit{Forward–backward stochastic differential equations and controlled McKean–Vlasov dynamics}. 
The Annals of Probability, 43(5):2647–2700, 2015.

\end{thebibliography}

\end{document}